\documentclass[pdflatex,sn-mathphys-num]{sn-jnl}
\usepackage{verbatim}
\usepackage{multirow}%
\usepackage{amsmath,amssymb,amsfonts}%
\usepackage{amsthm}%
\usepackage{mathrsfs}%
\DeclareMathSizes{8.43146}{8.43146}{6}{5}
\usepackage[title]{appendix}%
\usepackage{xcolor}%
\usepackage{textcomp}%
\usepackage{manyfoot}%
\usepackage{booktabs}%
\usepackage{algorithm}%
\usepackage{algorithmicx}%
\usepackage{algpseudocode}%
\usepackage{listings}%
\usepackage{bbm}
\let\orcidlogo\relax
\usepackage{orcidlink}
\theoremstyle{thmstyleone}%
\newtheorem{theorem}{Theorem}

\theoremstyle{thmstyletwo}%
\newtheorem{remark}{Remark}%
\newtheorem{lem}{Lemma}
\theoremstyle{thmstylethree}%
\newtheorem{definition}{Definition}%
\newtheorem{coro}{Corollary}
\begin{document}

\title{Relaxation equations with stretched non-local operators: renewals and time-changed processes}

\author*[1]{\fnm{Luisa} \sur{Beghin}}\email{luisa.beghin@uniroma1.it}\equalcont{These authors contributed equally to this work.}

\author[2]{\fnm{Nikolai} \sur{Leonenko}}\email{LeonenkoN@cardiff.ac.uk}
\equalcont{These authors contributed equally to this work.}

\author[3]{\fnm{Jayme} \sur{Vaz}}\email{vaz@unicamp.br}
\equalcont{These authors contributed equally to this work.}

\affil*[1]{%\orgdiv{Department},
\orgname{Sapienza University of Rome}, %\orgaddress{\street{Street}, \city{City}, \postcode{100190}, \state{State},
\country{Italy}}

\affil[2]{%\orgdiv{Department},
\orgname{Cardiff University}, %\orgaddress{\street{Street}, \city{City}, \postcode{10587}, \state{State},
\country{United Kingdom}}

\affil[3]{%\orgdiv{Department},
\orgname{Universidade Estadual de Campinas}, %\orgaddress{\street{Street}, \city{City}, \postcode{610101}, \state{State},
\country{Brazil}}

\abstract{
We introduce and study renewal processes defined by means of
extensions of the standard relaxation equation through ``stretched" non-local operators (of order $\alpha$ and with parameter $\gamma$). In a first case we obtain a generalization of the fractional Poisson process, which displays either infinite or finite expected waiting times between arrivals, depending on the parameter $\gamma$. Therefore, the introduction in the operator of the non-homogeneous term driven by $\gamma$ allows us to regulate the transition between different regimes of our renewal process.
We then consider a second-order relaxation-type equation involving the same operator,  under different sets of conditions on the constants involved; for a particular choice of these constants, we prove that the corresponding renewal process is linked to the first one by convex combination  of its distributions. We also discuss alternative models related to the same equations and their time-changed representation, in terms of the inverse of a non-decreasing process which generalizes the $\alpha$-stable L\'evy subordinator.

\vspace{0.5cm}
[AMS 2020]{ Primary: 60K05, 34A08. Secondary: 33E12, 26A33.}}

\keywords{stretched non-local operators, renewal processes, Kilbas-Saigo function, time-changed processes.}

\date{}

\maketitle

\section{Introduction}

We define here the following non-local operator
\begin{equation}
\label{eq.1}
\operatorname{\mathcal{D}}^{(\alpha,\gamma)}_t := t^{-\gamma} {\sideset{_{\scriptscriptstyle C}^{}}{_t^{(\alpha)}}{\operatorname{\mathcal{D}}}}, \qquad t >0,
\end{equation}
where  ${\sideset{_{\scriptscriptstyle C}^{}}{_t^{(\alpha)}}{\operatorname{\mathcal{D}}}} $ is the
Caputo derivative of order $\alpha$ (for $0 < \alpha < 1$) with respect to $t$, that is
\begin{equation}
{\sideset{_{\scriptscriptstyle C}^{}}{_t^{(\alpha)}}{\operatorname{\mathcal{D}}}}f(t) =
\frac{1}{\Gamma(1-\alpha)}\int_0^t \frac{f^\prime(\tau)}{(t-\tau)^\alpha}, d\tau , \qquad t >0, \notag
\end{equation}
which is defined for any absolutely continuous function, i.e. $f \in AC[0,t]$. The constant
$\gamma$ is a real parameter such that $\alpha + \gamma > 0$,   representing the ``stretching" parameter of the time argument.
We then analyze renewal processes defined by means of fractional equations which generalize the well-known relaxation equation
\begin{equation}
\frac{d}{dt}f(t)+\lambda f(t)=0, \qquad t > 0,  \lambda>0, \label{relax}
\end{equation}
with $f(0)=1,$ whose solution is the survival probability of the Poisson interarrival times (i.e. $f(t)=e^{-\lambda t}$, for $ t\geq 0$).

Stretched fractional operators are advanced mathematical tools used to model complex natural phenomena characterized by memory effects, non-local interactions, and fractal-like structures. In nature, these operators allow to model anomalous diffusion, heat and mass transfer, fluid dynamics,  complex biological systems, and quantum optics (see, e.g., \cite{LAS2},  \cite{Capelas1}, \cite{Capelas2}, \cite{POV}, \cite{VAZ}, and the references therein).

We start by considering a relaxation-type equation, where the first derivative in \eqref{relax} is replaced by the fractional operator  defined in \eqref{eq.1}, i.e.
\begin{equation}\label{uff}
\operatorname{\mathcal{D}}^{(\alpha,\gamma)}_tf(t)+\kappa f(t)=0,
\end{equation}
and we use the corresponding solution (expressed in terms of Kilbas-Saigo function) in order to define a renewal process $\mathcal{N}_{\alpha,\gamma}:=\left\{\mathcal{N}_{\alpha,\gamma}(t) \right\}_{t \geq 0}$.
Since, for $\gamma=0$, the operator $\mathcal{D}^{(\alpha,\gamma)}_t$ reduces to the Caputo type derivative $\mathcal{D}^{(\alpha)}_t$, we can consider this renewal process as an extension of the well-known time-fractional Poisson process (see, for example, \cite{LAS1}, \cite{MAI}, \cite{BEG2009}, \cite{MEE}).

Fractional extensions of the Poisson process in different directions can be found in \cite{KOC} (where general fractional calculus is introduced in the theory of renewals), \cite{GAR} (with a fractional derivative of variable order) and \cite{ALE} (in the case of fractional Poisson fields), among many others.

As we will see, when $\alpha+\gamma<1$, the survival function of the interarrival times of $\mathcal{N}_{\alpha,\gamma}$ coincides with the waiting time of the first event of the fractional non-homogeneous model introduced by Laskin in \cite{LAS}. This model is characterized by probability
mass function (hereafter p.m.f.) satisfying the following difference-differential equation
\begin{equation}\label{las}
    {\sideset{_{\scriptscriptstyle C}^{}}{_t^{(\alpha)}}{\operatorname{\mathcal{D}}}}p(n;t) + \lambda t^\gamma(p(n;t)-p(n-1;t)) = 0, \qquad n\in \mathbb{N}_0, t > 0,
\end{equation}
with initial condition $p(n;0)=\delta_0(n)$, $n=0,1,...$ and the convention that $p(j;t)=0$, for $j<0$ and for any $t$. In the special case where $\alpha=1$ and $\gamma=0$, the solution to the previous equation coincides with the p.m.f. of the Poisson process (with intensity $\lambda$), denoted by $\mathcal{N}:=\left\{\mathcal{N}(t)\right\}_{t \geq 0}$, i.e. $p(n;t)=e^{-\lambda t}\frac{(\lambda t)^n}{n!},$ $n \in \mathbb{N}_0$, $t \geq 0$.

%\begin{equation}\label{las}\operatorname{\mathcal{D}}^{(\alpha,\gamma)}_t p(n;t) + \lambda (p(n;t)-p(n-1;t)) = 0, \qquad n\in \mathbb{N}_0,t \geq 0,\end{equation}
It is proved in \cite{LAS} that the analytic solution to \eqref{las}
 is given in terms of successive derivatives of the Kilbas-Saigo (KS) function
 $\operatorname{E}_{\alpha,1+\gamma/\alpha,\gamma/\alpha}\left(-\lambda t^{\alpha+\gamma}\right)$ - see eq.\eqref{KS_function} -- i.e.
\begin{equation}p^L_{\alpha, \gamma}(n;t)=\frac{\left(\lambda t^{\alpha + \gamma}\right)^n}{n!}\operatorname{E}_{\alpha,1+\gamma/\alpha,\gamma/\alpha}^{(n)}\left(-\lambda t^{\alpha+\gamma}\right),\qquad n\in \mathbb{N}_0,t \geq 0, \label{pk}
\end{equation}
where
\begin{equation}
    \label{eq.4.a}
\operatorname{E}_{\alpha,1+\gamma/\alpha,\gamma/\alpha}^{(n)}(a):=\frac{d^n}{dx^n}\operatorname{E}_{\alpha,1+\gamma/\alpha,\gamma/\alpha}(x)\vert_{x=a}.
\end{equation}
Note that we will denote by the superscript $L$, both the solution \eqref{pk} and the corresponding process $\mathcal{N}^L_{\alpha,\gamma}:=\left\{\mathcal{N}^L_{\alpha,\gamma}(t)\right\}_{t \geq 0}$ defined by Laskin in \cite{LAS}, in order to distinguish it from those of our renewal model.

Some properties of the KS function, such as inequalities, complete monotonicity and spectral density can be found in \cite{BOU} and \cite{Capelas1}. In particular, since the KS function in \eqref{KS_function} is completely monotone (hereafter CM) for $\operatorname{Re} z<0$, $\alpha \in (0,1]$, $l \geq m-1/\alpha$ and $m>0$ (see Theorem 1 in \cite{Simon}), we will assume hereafter that $\alpha \in (0,1]$ and $\gamma>-\alpha$. Moreover, for $\alpha+\gamma \leq 1$, the p.m.f. given in \eqref{pk} is well-defined, since $\operatorname{E}_{\alpha,1+\gamma/\alpha,\gamma/\alpha}^{(n)}\left(-\lambda t^{\alpha+\gamma}\right) \geq 0$, for any $n \geq 0$. Nevertheless, since we are using a different approach than \cite{LAS}, we will consider also the case $\alpha+\gamma>1$; however, in the last case, no explicit analytical expression for the p.m.f. can be given.

We prove that for our model $\mathcal{N}_{\alpha,\gamma}$, when $\alpha+\gamma \leq 1$, the expected interarrival time between successive renewals is infinite (as for the fractional Poisson process, i.e. for $\gamma=0$), contrary to what happens for $\alpha+\gamma >1$.
%Moreover, the long-time behavior of the discrete-time incremental process $\mathcal{X}_{\alpha, \gamma}(n):=\mathcal{N}_{\alpha, \gamma}(n)-\mathcal{N}_{\alpha,\gamma}(n-1)$, for $n \in \mathbb{N},$ is characterized by long-range dependence in the first case and by anti-persistence in the second one.
As a consequence, the introduction in the operator \eqref{eq.1} of the non-homogeneous term driven by $\gamma$ permits us to regulate the transition between different regimes of our renewal process. A time-changed representation of the interarrival times is also derived in terms of the non-decreasing process $\mathcal{A}_{\alpha,\gamma}:=\left\{\mathcal{A}_{\alpha,\gamma}(t)\right\}_{t \geq 0}$ (see Theorem \ref{thmsub} below, for its definition). The latter is proved to generalize the $\alpha$-stable L\'{e}vy subordinator, to which it reduces, in the special case $\gamma=0$.

We then introduce a second-order term, involving the operator $\mathcal{D}^{(\alpha,\gamma)}_t$, in the relaxation-type equation \eqref{uff} as follows
\begin{equation}\label{uff2}
\left(\operatorname{\mathcal{D}}^{(\alpha,\gamma)}_t\right)^2 f(t) + a\operatorname{\mathcal{D}}^{(\alpha,\gamma)}_t f(t) +
b f(t) = 0 , \qquad t > 0,
\end{equation}
with $\alpha \in (0,1)$, $\alpha + \gamma > 0$, initial conditions $f(0) = \text{const}$ and $\operatorname{\mathcal{D}}^{(\alpha,\gamma)}_t f(t)|_{t=0}= \text{const}$,
$a,b \in \mathbb{R}^+$ and by $\left(\operatorname{\mathcal{D}}^{(\alpha,\gamma)}_t\right)^2$ we denote, for brevity, the sequential application of the operator in \eqref{eq.1} (see \cite{POD}, Sec. 2.5, for details on sequential fractional derivatives). We are interested in the solution to \eqref{uff2} as interarrival times' survival probability for an alternative renewal process.
Under some assumptions on $a$ and $b$, we prove that the latter
is linked to  $\mathcal{N}_{\alpha,\gamma}$ by convex combination  of its distributions.

Finally, we provide time-change representations for the process $\mathcal{N}^L_{\alpha,\gamma}$ and its second-order generalization (related to equation \eqref{uff2}), as a standard Poisson process time-changed by the inverse of the process $\mathcal{A}_{\alpha,\gamma}$, under an independence assumption.

\section{Preliminaries}

For the sake of completeness and ease of reading, in this section we present some definitions and results used in this work that are not widely known.

\subsection{The double gamma function}

The double gamma function, denoted by $G(z; \tau)$, is a
generalization of the Barnes $G$-function $G(z)$, which in turn is a
generalization of the gamma function $\Gamma(z)$. It is defined as \cite{Alexanian}%\cite{Kusnetsov}
\begin{equation}
\label{ap.B.G.tau.1}
\begin{split}
G(z;\tau) = & \frac{z}{\tau} \exp\left[a(\tau)\frac{z}{\tau} + b(\tau)\frac{z^2}{2\tau}\right] \\
& \cdot \prod_{m,n\in \mathbb{N}_\ast^2} \left[
\left(1+\frac{z}{m\tau+n}\right) \exp\left(-\frac{z}{m\tau+n}+\frac{z^2}{2(m\tau+n)^2}\right)\right]
\end{split}
\end{equation}
where $a(\tau)$ and $b(\tau)$ are given by
\begin{equation}\notag
\begin{split}
& a(\tau) = -\psi(1) \tau + \frac{\tau}{2}\log{(2\pi \tau)} + \frac{1}{2}\log\tau - \tau C(\tau) , \\
& b(\tau) = - \frac{\pi^2 \tau^2}{6} - \tau \log\tau - \tau^2 D(\tau) ,
\end{split}
\end{equation}
with
\begin{equation}\notag
\begin{split}
& C(\tau) = \lim_{m\to \infty} \left[\sum_{k=1}^{m-1}\psi(k\tau) + \frac{1}{2}\psi(m\tau) - \frac{1}{\tau}
\log\left(\frac{\Gamma(m\tau)}{\sqrt{2\pi}}\right) \right] , \\
& D(\tau) = \lim_{m\to \infty} \left[\sum_{k=1}^{m-1}\psi^\prime(k\tau) + \frac{1}{2}\psi^\prime(m\tau) - \frac{1}{\tau}\psi(m\tau) \right] ,
\end{split}
\end{equation}
where $\psi(\cdot)$ is the digamma function (and so $-\psi(1)$ is the Euler-Mascheroni constant).
Another definition of double gamma function is
\begin{equation}
\label{ap.B.G.tau.2}
G(z;\tau) = \frac{1}{\tau\Gamma(z)} \exp\left[\tilde{a}(\tau)\frac{z}{\tau} + \tilde{b}(\tau)
\frac{z^2}{2\tau^2}\right] \prod_{m=1}^\infty \frac{\Gamma(m\tau)}{\Gamma(z+m\tau)}
\exp\left[z\psi(m\tau)+ \frac{z^2}{2}\psi^\prime(m\tau) \right] ,
\end{equation}
where
\begin{equation}\notag
\tilde{a}(\tau) = a(\tau) - \gamma\tau , \qquad \tilde{b}(\tau) = b(\tau) + \frac{\pi^2 \tau^2}{6} .
\end{equation}
The Barnes $G$-function $G(z)$ is a particular case of the double
gamma function $G(z;\tau)$ for $\tau = 1$. Note that $C(1) = 1/2$ and $D(1) = 1+ \gamma$, as shown in \cite{Genesis}.

The double gamma function is an entire function and its zeros are located
at $z = z_{mn} = -m\tau-n$ ($m,n = 0,1,2,\ldots$).
It is defined in such a way that \cite{Genesis}
\begin{equation}
\label{ap.B.G.tau.(1)}
G(1;\tau) = 1 .
\end{equation}
It satisfies the functional relations
\begin{equation}
\label{ap.B.general.G.tau}
G(z+1;\tau) = \Gamma(z/\tau) G(z;\tau)
\end{equation}
and
\begin{equation}
\label{ap.B.general.G.tau.2}
G(z+\tau;\tau) = (2\pi)^{\frac{\tau-1}{2}} \tau^{-z+\frac{1}{2}} \Gamma(z) G(z;\tau) .
\end{equation}
Using \eqref{ap.B.general.G.tau} recursively, we obtain
\begin{equation}
\label{ap.B.general.G.tau.n}
G(z+k;\tau) = G(z,\tau) \prod_{j=0}^{k-1} \Gamma[(z+j)/\tau] .
\end{equation}

The Stirling formula for $G(z;\tau)$ is \cite{Alexanian}
\begin{equation}
\label{ap.B.Stirling.G.tau}
\log{G(z;\tau)} = \left[a_2(\tau) z^2 + a_1(\tau)z + a_0(\tau)\right] \log{z} +
b_2(\tau)z^2 + b_1(\tau) z + b_0(\tau) + \mathcal{O}(z^{-1}) ,
\end{equation}
where
\begin{equation}\notag
\begin{split}
& a_0(\tau) = \frac{\tau}{12} + \frac{1}{4} + \frac{1}{12\tau} , \\
& a_1(\tau) = -\frac{1}{2}\left(1+\frac{1}{\tau}\right) , \\
& a_2(\tau) = \frac{1}{2\tau} ,
\end{split}
\end{equation}
and
\begin{equation}\notag
\begin{split}
& b_2(\tau) = -\frac{1}{2\tau}\left(\frac{3}{2} + \log\tau\right) , \\
& b_1(\tau) = \frac{1}{2}\left(\left(1+\frac{1}{\tau}\right)(1+\log\tau) + \log{2\pi}\right) , \\
& b_0(\tau) = \frac{1}{3}\big\{\log\left[G^2(1/2;\tau)G(\tau;2\tau)\right] \\
& \phantom{b_0(\tau) = \frac{1}{3}\big\{} -
\frac{1+\tau}{2}\log{2\pi} - a_0(\tau)\log{(\tau^3/2)} - \log{2}\big\} .
\end{split}
\end{equation}

An integral representation for $G(z;\tau)$ was provided in \cite{lawrie}, that is,
\begin{equation}
\label{double.gamma.int.rep}
\begin{aligned}
\ln G(z;\tau) & = \int_0^1 \bigg[ \frac{r^{z-1}}{(r-1)(r^\tau-1)} -
\frac{z^2}{2\tau} r^{\tau-1} - z r^{\tau-1}\left(\frac{2-r^\tau}{r^\tau-1}-\frac{1}{2\tau}\right) \\
& -r^{\tau-1} + \frac{1}{r-1} - \frac{r^{\tau-1}}{(r-1)(r^\tau-1)}\bigg] \frac{dr}{\ln{r}} .
\end{aligned}
\end{equation}
It is convergent for $\operatorname{Re}z > 0$ and $\tau > 0$.

More details about the double gamma function can be found in \cite{Alexanian}
and \cite{SEFC2nd}.

\subsection{Kilbas-Saigo function}
\label{sub_sec_KS}

 The KS function is defined as (see \cite{Simon,KS.book})
 \begin{equation}
\label{KS_function}
\operatorname{E}_{a,m,l}(z): = 1  + \sum_{n=1}^\infty \left( \prod_{k=1}^n \frac{\Gamma(1+a((k-1)m+l))}{\Gamma(1+a((k-1)m+l+1))}\right) z^n=:  1  + \sum_{n=1}^\infty c_n z^n,
\end{equation}
for $a, m > 0$, $l > -1/a$ and $a(km+l) \neq -1,-2,...$, for any $k\geq 1$,
and $z\in \mathbb{C}$. The KS is an entire function.

Using eq.\eqref{ap.B.general.G.tau.n} we can
 rewrite the coefficients of the KS function
in terms of the double gamma function $G(z;\tau)$.
We obtain:
\begin{equation*}
	\label{coeff.KS.G.pochhammer}
	c_n
    %= \frac{1}{[\mathsf{n}!]_\alpha^{\alpha+\gamma}}
    = \frac{G(\varphi + a\tau;\tau)}{G(\varphi ;\tau)} \, \frac{G(\varphi   + n;\tau)}{G(\varphi + a\tau+ n;\tau)} ,
\end{equation*}
where, for $n=1,2,\ldots$,
\begin{equation}
	\label{def.varphi}
	\varphi = (1+al)\tau
\end{equation}
and
\begin{equation}
	\label{exp.tau}
	\tau = \frac{1}{am}.
\end{equation}
Note that this expression also works for $n=0$.
Thus we can write
\begin{equation}
	\label{def.KS.G.pochhammer}
	\operatorname{E}_{a,m,l}(z) = \frac{G(\varphi + a\tau;\tau)}{G(\varphi ;\tau)} \sum_{n=0}^\infty  \frac{G(\varphi   + n;\tau)}{G(\varphi + a\tau+ n;\tau)} z^n .
\end{equation}

Some interesting particular cases of the KS are
\begin{equation}
\operatorname{E}_{a,1,0}(z) = \sum_{n=0}^\infty \frac{z^n}{\Gamma(n a + 1)} = \operatorname{E}_a(z) ,
\end{equation}
which is the one-parameter Mittag-Leffler function, and
\begin{equation}
\label{ML_2_param}
\operatorname{E}_{a,1,(b-1)/a}(z) = \Gamma(b) \sum_{n=0}^\infty
\frac{z^n}{\Gamma(a n + b)} = \Gamma(b) \operatorname{E}_{a,b}(z) ,
\end{equation}
where $\operatorname{E}_{a,b}(z)$ is the two-parameter Mittag-Leffler function,
with $a,b > 0$ and $z\in \mathbb{C}$.

One important property of the KS function is the complete monotonicity
of $f(x) = \operatorname{E}_{a,m,l}(-x)$ for $\alpha \in (0,1]$,
$l \geq m -1/a$ and $x\geq 0$ \cite{Simon}.  On the other hand, it is known (see \cite{SCH}, Theorem 3.7) that if $f(\cdot)$ is a completely monotone function and $g(\cdot)$ is a Bernstein function, then $f(g(\cdot))$ is also completely monotone. Since the function $g(x) = \lambda x^\rho$, for $0\leq \rho \leq 1$ and $\lambda > 0$, is a Bernstein function \cite{SCH}, then $f(g(x)) = \operatorname{E}_{a,m,l}(-\lambda x^\rho)$ is completely monotone in this range of parameters. For $\rho > 1$, however, the function $g(\cdot)$ ceases to be a Bernstein function, and so $f(g(\cdot))$ is not completely monotone when $\rho > 1$. Nevertheless,
the function $f(g(\cdot)) $  remains monotonically decreasing. In fact, since $f^\prime(y) < 0$ for $y > 0$, then $f^\prime(\lambda x^\rho) < 0$ for $x>0$, and so
$ [f(g(x))]^\prime = f^\prime(g(x)) \rho x^{\rho -1} < 0$
for $x > 0$.

A useful result is \cite{Simon}
\begin{equation}
\frac{1}{1+\Gamma(1-a)x} \leq \operatorname{E}_{a,m,m-1}(-x) \leq
\frac{1}{1+\frac{\Gamma(1 + a(m-1)}{\Gamma(1+am)}x} ,
\end{equation}
where $a\in [0,1]$, $m>0$ and $x\geq 0$.

\subsection{Fibonacci polynomials}

The Fibonacci polynomials $F_n(x)$ ($n=0,1,2,\ldots$) are defined by the recursive relation
\begin{equation}
\label{fibo.1}
    F_{n+2}(x) = x F_{n+1}(x) + F_n(x), \qquad F_0(x) = 0 , \quad F_1(x) = 1 .
\end{equation}
An alternative expression is
\begin{equation}
\label{fibo.2}
F_n(x) = \frac{[\mu(x)]^n - [\nu(x)]^n}{\mu(x) - \nu(x)} ,
\end{equation}
where
\begin{equation}
\label{mu.nu}
\mu(x) = \frac{x+\sqrt{x^2+4}}{2} , \qquad
\nu(x) = \frac{x-\sqrt{x^2 +4}}{2}
\end{equation}
(see \cite{Fibonacci})

The bivariate
Fibonacci polynomials $u_n(x,y)$ are defined by \cite{Hoggatt} as
\begin{equation}
\label{bivariate.fib}
u_{n+2}(x,y) = x u_{n+1}(x,y) + y u_n(x,y) , \qquad u_0(x,y) = 0, \quad u_1(x,y) = 1.
\end{equation}
An alternative expression is
\begin{equation}
u_n(x,y) = \frac{[\bar{\mu}(x,y)]^n - [\bar{\nu}(x,y)]^n}{\bar{\mu}(x,y)-\bar{\nu}(x,y)}
\end{equation}
where
\begin{equation}
\label{bar.mu.nu}
\bar{\mu}(x,y) = \frac{x+\sqrt{x^2+4y}}{2} , \qquad
\bar{\nu}(x,y) = \frac{x-\sqrt{x^2 +4y}}{2} .
\end{equation}
Note that we have
\begin{equation}
\label{fib.aux}
u_n(x,y) = y^{(n-1)/2} F_n(x/\sqrt{y}) .
\end{equation}

\section{A first-order equation involving $\mathcal{D}^{(\alpha,\gamma)}_t$}     \label{sec1}
We recall that
\begin{equation*}
{\sideset{_{\scriptscriptstyle C}^{}}{_t^{(\alpha)}}{\operatorname{\mathcal{D}}}} t^\beta  =
\begin{cases}
{\displaystyle \frac{\Gamma(\beta+1)}{\Gamma(\beta-\alpha+1)} t^{\beta-\alpha}} , \; & \beta \neq 0 , \\[1ex]
0 , & \beta = 0 ,
\end{cases}
\end{equation*}
(see \cite{KIL}) and thus
\begin{equation}\label{dag}
\operatorname{\mathcal{D}}^{(\alpha,\gamma)}_t  \, t^\beta
=
\begin{cases}
{\displaystyle \frac{\Gamma(\beta+1)}{\Gamma(\beta-\alpha+1)} t^{\beta-(\alpha+\gamma)}} , \; & \beta \neq 0 , \\[1ex]
0 , & \beta = 0 .
\end{cases}
\end{equation}

Let us now consider the fractional differential equation
\begin{equation}
\label{fde.1}
\operatorname{\mathcal{D}}^{(\alpha,\gamma)}_t f(t) + \kappa f(t) = 0, \qquad t > 0,
\end{equation}
with initial condition $f(0) = f_0$,
for $\kappa \in \mathbb{R}^+$,  $\alpha \in (0,1)$ and $\alpha + \gamma > 0 $, which can be used as a model for anomalous relaxation (see \cite{Capelas1,Capelas2}).

\medskip

\paragraph{Preliminary results}

 We look for solutions of eq.(\ref{fde.1}) of the following form
\begin{equation}
\label{eq.3}
f(t) = \sum_{n=0}^\infty f_n t^{(\alpha+\gamma)n} .
\end{equation}

Let us introduce the compact notation
\begin{equation}
\label{def.n.beta.alpha}
[ \mathsf{n}]^\beta_\alpha := \frac{\Gamma(\beta n + 1)}{\Gamma(\beta n - \alpha + 1)}
\end{equation}
and
\begin{equation*}
[ {\mathsf{n}_1 \times  \mathsf{n}_2 \times  \cdots \times  \mathsf{n}_k} ]^\beta_\alpha :=
[ \mathsf{n}_1 ]^\beta_\alpha[ \mathsf{n}_2 ]^\beta_\alpha\cdots
[ \mathsf{n}_k ]^\beta_\alpha.
\end{equation*}
Moreover, let
\begin{equation}
\label{def.n.beta.alpha.factorial}
[\mathsf{n!}]^\beta_\alpha := [\mathsf{n}\times (\mathsf{n-1})\times \cdots \times \mathsf{1}]^\beta_\alpha .
\end{equation}
Thus we have that
\begin{equation}
\label{der.order.1}
\left(\operatorname{\mathcal{D}}^{(\alpha,\gamma)}_t\right) f(t) = \sum_{n=0}^\infty
f_{n+1}[ \mathsf{n+1}]^{\alpha+\gamma}_\alpha t^{(\alpha+\gamma)n}
\end{equation}
and, from eq.(\ref{fde.1}),
\begin{equation}\label{derdernow}
f_{n+1}[ \mathsf{n+1}]^{\alpha+\gamma}_\alpha + \kappa f_n = 0 , \qquad n=0,1,2,\ldots .
\end{equation}
Thus the solution reads
\begin{equation}
\label{sol.fde1}
f(t) = \sum_{n=0}^\infty \frac{(-\kappa t^{\alpha+\gamma})^n}{[\mathsf{n!}]^{\alpha+\gamma}_\alpha} f_0 ,
\end{equation}
where we define $[\mathsf{0!}]^{\alpha+\gamma}_\alpha = 1$. This corresponds to a KS function with
\begin{equation}
a = \alpha, \qquad m = 1 + \frac{\gamma}{\alpha} , \qquad l = \frac{\gamma}{\alpha} , \notag
\end{equation}
that is
\begin{equation}\label{KS}
f(t) = f_0 \, \operatorname{E}_{\alpha,1+\gamma/\alpha,\gamma/\alpha}\left(-\kappa t^{\alpha+\gamma}\right).
\end{equation}

The uniform and absolute convergence of the series in \eqref{sol.fde1} as well as the uniqueness of the solution of eq.\eqref{fde.1}, in the space of absolutely continuous functions, are discussed in Appendix~\ref{appendix.A}.

\paragraph{Particular case}  When $\gamma = 0$, we have that $\mathcal{D}^{(\alpha,\gamma)}_t =
{\sideset{_{\scriptscriptstyle C}^{}}{_t^{(\alpha)}}{\operatorname{\mathcal{D}}}}$ and
\begin{equation*}
[\mathsf{n}]^{\alpha+\gamma}_\alpha = [\mathsf{n}]_\alpha^\alpha =
\frac{\Gamma(\alpha n+1)}{\Gamma(\alpha(n-1)+1)} ,
\end{equation*}
so that $[\mathsf{n!}]^{\alpha}_\alpha = \Gamma(\alpha n + 1).$
Thus
\begin{equation}
\operatorname{E}_{\alpha,1,0}\left(-\kappa t^{\alpha}\right) =
\sum_{n=0}^\infty \frac{(-\kappa t^{\alpha})^n}{\Gamma(\alpha n + 1)} =: \operatorname{E}_\alpha(-\kappa
t^\alpha) , \label{fpp}
\end{equation}
where $\operatorname{E}_\alpha(\cdot)$ is the Mittag-Leffler function; this agrees with the result given in \cite{MAI}, where it is shown that the function \eqref{fpp} satisfies the time-fractional relaxation equation (with the standard Caputo derivative).
 \medskip

\subsection{First renewal model}

We start by recalling the following probabilistic representation of the Kilbas-Saigo function (see \cite{BOU}), for $\alpha \in (0,1)$ and $\alpha + \gamma > 0$:
\begin{equation}
\operatorname{E}_{\alpha,1+\gamma/\alpha,\gamma/\alpha}\left(-\lambda t^{\alpha+\gamma}\right)=\mathbb{E}e^{-\lambda t^{\alpha+\gamma}\int_{0}^\infty (1-\mathcal{A}_\alpha(s))_+^{\gamma}ds}, \qquad \lambda>0, t \geq 0, \label{LTZ}
\end{equation}
where $\mathcal{A}_\alpha:=\left\{ \mathcal{A}_\alpha(t)\right\}_{t \geq 0}$ is the $\alpha$-stable subordinator, i.e. a non-decreasing L\'{e}vy process with $\mathbb{E}e^{-\kappa\mathcal{A}_\alpha(t)}=e^{-t\kappa^\alpha }$, for $t \geq 0$, $\kappa >0$ and $\alpha \in (0,1)$.  Thus, we can represent the solution to equation \eqref{fde.1}, for $\kappa=\lambda$ and $f_0=1,$ as
\begin{equation}
f(t)=\mathbb{E}e^{-\lambda \mathcal{Z}_{\alpha,\gamma}(t)},  \qquad \lambda>0, t \geq 0, \label{ff}
\end{equation}
where the random process $\left\{\mathcal{Z}_{\alpha,\gamma}(t) \right\}_{t \geq 0}$ is defined as \begin{equation}\mathcal{Z}_{\alpha,\gamma}(t):=t^{\alpha+\gamma}Z_{\alpha,\gamma}, \qquad t \geq0, \label{defz} \end{equation}
for the r.v. $Z_{\alpha,\gamma}:=\int_{0}^\infty (1-\mathcal{A}_\alpha(s))_+^{\gamma}ds$. Moreover, the following equality in distribution is proved in \cite{BOU} for $Z_{\alpha,\gamma}$, in terms of an infinite independent product of Beta r.v.'s
$B(a,b)$:

\begin{equation}
Z_{\alpha,\gamma} \overset{d}{=}\frac{\Gamma(\gamma+1)}{\Gamma(\alpha+\gamma+1)}\prod_{n=0}^\infty \frac{\gamma +n+1}{\alpha+\gamma+n}B\left(1+\frac{n}{\alpha+\gamma},\frac{1-\alpha}{\alpha+\gamma}\right). \label{beta}\end{equation}
Details on the a.s. convergence of the infinite product in \eqref{beta} (and of its Mellin transform) can be found in \cite{LET}.

\begin{remark}
Let us recall the definition of the right-continuous inverse (or first-passage time) of the stable subordinator, which we denote by $\mathcal{L}_\alpha:=\left\{  \mathcal{L}_\alpha(t)\right\}_{t \geq 0}$, i.e.
\[
\mathcal{L}_\alpha(t):=\inf\left\{s>0:\mathcal{A}_\alpha(s)>t \right\}, \qquad t\geq 0, \alpha \in (0,1).
\]
Then, it is well-known that the Laplace transform (hereafter LT) of its density reads $\mathbb{E}e^{-\lambda \mathcal{L}_\alpha(t)}=E_{\alpha, 1}(-\lambda t^\alpha)$, which coincides with $\mathbb{E}e^{-\lambda t^{\alpha}\mathcal{L}_\alpha(1)}$, by the self-similarity property of $\mathcal{L}_{\alpha}$ (see \cite{MEE2}). Thus, we can conclude that, in the special case where $\gamma=0$ and $\alpha \in (0,1)$, the following equality in distribution holds: $Z_{\alpha,0} \overset{d}{=}\mathcal{L}_{\alpha}(1)$.
\end{remark}
Then, we give the following definition.
\begin{definition}\label{defren}
    Let $\alpha \in (0,1)$, $\alpha + \gamma > 0$ and let $\mathcal{N}_{\alpha, \gamma}:=\left\{\mathcal{N}_{\alpha,\gamma}(t) \right\}_{t \geq 0}$ be the renewal process with i.i.d. interarrival times $U^{(\alpha, \gamma)}_{j}$, $j=1,2,...$, satisfying \eqref{fde.1} with $\kappa=\lambda >0$ and $f_0=1$, i.e. with survival probability
    \begin{equation}\label{defdef}
     \mathbb{P}(U^{(\alpha, \gamma)}>t)=\operatorname{E}_{\alpha,1+\gamma/\alpha,\gamma/\alpha}\left(-\lambda t^{\alpha+\gamma}\right),
    \end{equation}
    by \eqref{KS}. Then we denote, for any $t \geq 0,$ \[\mathcal{N}_{\alpha, \gamma}(t):=\sup \left\{ n \geq 1: T_n^{(\alpha,\gamma)} \leq t \right\}=\sum_{n=0}^\infty n\mathbbm{1}_{\left\{T_n^{(\alpha,\gamma)} \leq t < T_{n+1}^{(\alpha,\gamma)}\right\}},\] where
$T_n^{(\alpha,\gamma)}:=\sum_{j=1}^n U^{(\alpha, \gamma)}_{j}$ and $\mathbbm{1}_{\left\{\cdot\right\}}$ is the indicator function.

    \end{definition}

As seen in Section~\ref{sub_sec_KS}, the function $f(x):=\operatorname{E}_{\alpha,1+\gamma/\alpha,\gamma/\alpha}(-x)$ for $x>0$
is CM when $\alpha + \gamma \leq 1$, but not for $\alpha + \gamma > 1$, although it still decreases monotonically for $\alpha + \gamma > 1$.
%%\begin{remark}\label{Remark3}Note that, by the complete monotonicity of ${\color{red}f(x):=}\operatorname{E}_{\alpha,1+\gamma/\alpha,\gamma/\alpha}(-x)$, for $x>0$, under the  assumptions recalled above, the function in \eqref{defdef} is CM only for $\alpha+\gamma <1$ (which is the range of the parameters  considered in \cite{LAS}), since it is the composition of a CM function and a Bernstein one, in this case (see \cite{SCH}, Theorem 3.7). Indeed, {\color{red}$g(t):=t^{\alpha+\gamma}$ is Bernstein for $\alpha+\gamma<1$, so that the Fa\'{a} di Bruno formula gives \begin{equation}(-1)^n\frac{d^n}{dt^n}f(g(t))=\sum_{(m,i_1,...,i_k)}\frac{n!}{i_1!\dots i_k!}(-1)^mf^{(m)}(g(t)) \prod_{j=1}^k \left( \frac{(-1)^{j-1}g^{(j)}(t)}{j!}\right)^{i_j}\geq0, \label{rem3}
%\end{equation}
%where the sum is over all $k \in \mathbb{N}$ and $i_1, \dots , i_k \in \mathbb{N}\cup\left\{0\right\}$ such that $\sum_{j=1}^ki_j=m$ and $\sum_{j=1}^kji_j=n$.
% On the other hand, for $\alpha+\gamma>1$, the product in \eqref{rem3} is negative (at least for $k=2$); however }the function in \eqref{defdef} is still monotonically decreasing, since
%\begin{equation}
%    \frac{d}{dt}\operatorname{E}_{\alpha,1+\gamma/\alpha,\gamma/\alpha}(-\lambda t^{\alpha+\gamma}) =\left.(\alpha+\gamma)\lambda t^{\alpha+\gamma-1}\frac{d}{dx}{E}_{\alpha,1+\gamma/\alpha,\gamma/\alpha}(-x) \right\vert_{x=\lambda t^{\alpha+\gamma}} <0.
%\end{equation}
In both cases, we have that $ \mathbb{P}(U^{(\alpha, \gamma)}>0)=1$ and $\lim_{t \to +\infty} \mathbb{P}(U^{(\alpha, \gamma)}>t)=0$,
by referring to the asymptotic expansion of the KS function:
\begin{equation}\label{exp}
\operatorname{E}_{a,m,l}(-z)\sim  \frac{\Gamma(1+a(l-m+1))}{\Gamma(1+a(l-m))}  z^{-1},
\end{equation}
as $z \to + \infty$, proved in \cite{Simon}, Proposition 6.

Thus the function in \eqref{defdef} is a proper survival probability (see also \cite{Simon}, Theorem 1.1) and the process $\mathcal{N}_{\alpha, \gamma}$ is always well defined, according to Def. \ref{defren}.
%\end{remark}

\begin{theorem}\label{rem2.3}
Let $\alpha \in (0,1)$ and $\alpha + \gamma > 0$. Then for  $\alpha+\gamma  \leq 1$, the expected value of the interarrival time $U^{(\alpha, \gamma)}$ of the process $\mathcal{N}_{\alpha,\gamma}$ is infinite, whereas, for $\alpha+\gamma>1$, it is finite and reads
 \begin{equation}
 \label{eu}\mathbb{E}U^{(\alpha, \gamma)}= \frac{1}{\lambda^{1/(\alpha+\gamma)}(\alpha+\gamma-1)\Gamma(1-\alpha)}.
 \end{equation}

 \end{theorem}

 \begin{proof}
    We recall the asymptotic behavior of the Mellin transform of the KS function as
proved in \cite{Simon}:
 \begin{equation}\label{ut2}
\int_0^{+\infty}\operatorname{E}_{a,m,l}(-x)x^{s-1}dx \sim \frac{\Gamma(1+a(l-m+1))}{\Gamma(1+a(l-m))(1-s)}, \qquad s \uparrow 1,
 \end{equation}
 for $a  \in (0,1)$, $m>0$ and $l>m-1/a$.

Thus, for $s \in (0,1)$, we can write that
 \begin{eqnarray}
\int_0^{+\infty}\mathbb{P}(U^{(\alpha, \gamma)}>t)t^{s-1}dt &=&\frac{1}{(\alpha+\gamma)\lambda^{1/(\alpha+\gamma)}}\int_0^{+\infty}\operatorname{E}_{\alpha,1+\gamma/\alpha,\gamma/\alpha}(-x)x^{\frac{s}{\alpha+\gamma}-1}dx \notag \\
&\geq& \frac{1}{(\alpha+\gamma)\lambda^{1/(\alpha+\gamma)}}\int_0^{+\infty}\operatorname{E}_{\alpha,1+\gamma/\alpha,\gamma/\alpha}(-x)x^{s-1}dx \notag \\
&\sim& \frac{1}{(\alpha+\gamma)\lambda^{1/(\alpha+\gamma)}(1-s)\Gamma(1-\alpha)}, \qquad s \uparrow 1, \notag
\end{eqnarray}
 where the inequality holds only for $\alpha+\gamma<1$. As a consequence, by considering that $\mathbb{E}U^{(\alpha, \gamma)}= \lim_{s \uparrow 1}\int_0^{+\infty}\mathbb{P}(U^{(\alpha, \gamma)}>t)t^{s-1}dt$, the result follows for $\alpha+\gamma<1$.
 In the other case, for $\alpha+\gamma >1,$ we have, instead, that $s/(\alpha+\gamma)<1$, so that we get
 \[
 \mathbb{E}U^{(\alpha, \gamma)} =\lim_{s \uparrow 1} \frac{1}{\lambda^{1/(\alpha+\gamma)}(\alpha+\gamma-s)\Gamma(1-\alpha)} < \infty
 \]
and coincides with \eqref{eu}.

When $\alpha + \gamma = 1$ we have
 \begin{eqnarray}
\int_0^{+\infty}\mathbb{P}(U^{(\alpha, \gamma)}>t)t^{s-1}dt &=&\frac{1}{\lambda}\int_0^{+\infty}\operatorname{E}_{\alpha,1+\gamma/\alpha,\gamma/\alpha}(-x)x^{s-1}dx \notag \\
&\sim& \frac{1}{\lambda (1-s)\Gamma(1-\alpha)}, \qquad s \uparrow 1. \notag
\end{eqnarray}
 \end{proof}

\begin{remark}\label{pro}
We show that our model does not coincide with the process introduced and studied by Laskin in \cite{LAS}, which will be indicated by $\mathcal{N}^L_{\alpha,\gamma}$. Moreover, since they share the distribution of the first arrival's waiting time, this proves that the model defined in \cite{LAS} is not a renewal process.
As a consequence of Def. \ref{defren}, and by considering that the following relationship holds for any renewal process
\begin{equation}\label{reln}
\left\{\mathcal{N}_{\alpha,\gamma}(t) \geq n\right\}=\left\{T^{(\alpha,\gamma)}_n \leq t\right\}=\left\{\sum_{j=1}^n U_j^{(\alpha,\gamma)}\leq t\right\},
\end{equation}
we can show that the p.m.f of $\mathcal{N}_{\alpha, \gamma}$, i.e., $\left\{p_{\alpha, \gamma}(n;t)  \right\}_{n \in \mathbb{N}_0}$, with $p_{\alpha, \gamma}(n;t):=\mathbb{P}(\mathcal{N}_{\alpha, \gamma}=n)$, does not satisfy the equation studied in \cite{LAS} and recalled in \eqref{las}.

Indeed, if we denote by $(f*g)(t):=\int^t_0 f(z)g(t-z)dz$ the convolution of two functions $f,g:\mathbb{R}^+ \to \mathbb{R}$, %and $f^{*(n)}(\cdot):=(f^{*(n-1)}*f)(\cdot)$
we easily obtain from \eqref{reln} that, for any $n \in \mathbb{N}$ and $t \geq 0,$ the p.m.f. of a renewal process satisfies
\begin{equation}\label{psi}
p_{\alpha,\gamma}(n;t)=(f_{T^{(\alpha, \gamma)}_n}*\Psi)(t), \qquad n \in \mathbb{N}, t \geq 0,
\end{equation}
where $\Psi(t):=\mathbb{P}(U^{(\alpha,\gamma)}>t)$ and $f_{T^{(\alpha, \gamma)}_n}(\cdot)$ is the density function of the $n$-th event's waiting time. Therefore, we have that
\begin{eqnarray}
&& t^{-\gamma} {\sideset{_{\scriptscriptstyle C}^{}}{_t^{(\alpha)}}{\operatorname{\mathcal{D}}}}\int_0^t f_{T^{(\alpha, \gamma)}_n}(z)\Psi(t-z)dz  \notag \\
&=& \frac{t^{-\gamma}\Psi(0)}{\Gamma(1-\alpha)}\int_0^t \frac{d\tau}{(t-\tau)^\alpha}f_{T^{(\alpha, \gamma)}_n}(\tau)+\frac{t^{-\gamma}}{\Gamma(1-\alpha)}\int_0^t \frac{d\tau}{(t-\tau)^\alpha}\int_0^\tau f_{T^{(\alpha, \gamma)}_n}(z) \frac{d}{d\tau}\Psi(\tau-z)dz \notag \\
&=:& I+II. \notag
\end{eqnarray}
Then by Def. \ref{defren} and by considering that $f_{U^{(\alpha,\gamma)}}(t)=-\Psi'(t),$ we can write
\begin{eqnarray}
I&=&\frac{t^{-\gamma}}{\Gamma(1-\alpha)}\int_0^t \frac{d\tau}{(t-\tau)^\alpha}(f_{T^{(\alpha, \gamma)}_{n-1}}*f_{U^{(\alpha,\gamma)}})(\tau) =-t^{-\gamma}\int_0^t  f_{T^{(\alpha, \gamma)}_{n-1}}(z) ({\sideset{_{\scriptscriptstyle C}^{}}{_{\cdot}^{(\alpha)}}{\operatorname{\mathcal{D}}}}\Psi)(t-z)dz\notag \\
&=&-t^{-\gamma}\int_0^t  f_{T^{(\alpha, \gamma)}_{n-1}}(z) (t-z)^\gamma(\operatorname{\mathcal{D}}^{(\alpha,\gamma)}_{\cdot} \Psi)(t-z)dz=\lambda t^{-\gamma}\int_0^t  f_{T^{(\alpha, \gamma)}_{n-1}}(z) (t-z)^\gamma \Psi(t-z)dz .\notag
%\\&\neq&\lambda(f_{T^{(\alpha, \gamma)}_{n-1}}*\Psi)(t) =\lambda p_{\alpha,\gamma}(n-1;t). \notag
\end{eqnarray}
Analogously, for the second term, we have that
\begin{eqnarray}
 II&=&\frac{t^{-\gamma}}{\Gamma(1-\alpha)}\int_0^t f_{T^{(\alpha, \gamma)}_n}(z) dz \int_0^{t-z} \frac{d\tau}{(t-z-w)^\alpha}  \frac{d}{dw}\Psi(w)dw
 \notag \\
&=&-\lambda t^{-\gamma}\int_0^t  f_{T^{(\alpha, \gamma)}_{n}}(z) (t-z)^\gamma \Psi(t-z)dz  \notag
\end{eqnarray}
and thus
\begin{eqnarray}I+II&=&\lambda t^{-\gamma}\int_0^t  \left[f_{T^{(\alpha, \gamma)}_{n-1}}(z)-f_{T^{(\alpha, \gamma)}_{n}}(z)\right] (t-z)^\gamma \Psi(t-z)dz \notag \\
&\neq&\lambda\left(\left[f_{T^{(\alpha, \gamma)}_{n-1}}-f_{T^{(\alpha, \gamma)}_{n}}\right]*\Psi\right)(t) =\lambda (p_{\alpha,\gamma}(n-1;t)-p_{\alpha,\gamma}(n;t)).\notag\end{eqnarray}

The conclusion follows by considering that $f_{T^{(\alpha, \gamma)}_{n-1}}-f_{T^{(\alpha, \gamma)}_{n}}$ is non-negative (since $T_n \geq T_{n-1}$ a.s.) and $\Psi(\cdot) \geq 0$, so that $$\int_0^t  \left[f_{T^{(\alpha, \gamma)}_{n-1}}(z)-f_{T^{(\alpha, \gamma)}_{n}}(z)\right] (1-z/t)^\gamma \Psi(t-z)dz \neq \left(\left[f_{T^{(\alpha, \gamma)}_{n-1}}-f_{T^{(\alpha, \gamma)}_{n}}\right]*\Psi\right)(t),$$ unless $\gamma=0.$

\end{remark}

\paragraph{Stochastic representation}

In order to derive a time-change characterization of the interarrival times of $\mathcal{N}_{\alpha,\gamma}$, we start by proving that the process $\mathcal{Z}_{\alpha, \gamma}:=\left\{\mathcal{Z}_{\alpha, \gamma}(t)\right\}_{t \geq 0}$,  defined in equation \eqref{defz}, is equal in distribution to the inverse of an almost-surely increasing process, which will be denoted as $\mathcal{A}_{\alpha, \gamma}:=\left\{\mathcal{A}_{\alpha, \gamma}(t)\right\}_{t \geq 0}$, starting a.s. from $0,$ and such that the following relationship
\begin{equation}\label{rel}
\mathbb{P}\left( \mathcal{A}_{\alpha,\gamma}(t)<x\right)=\mathbb{P}\left( \mathcal{Z}_{\alpha,\gamma}(x)>t\right)
\end{equation}
holds, for any $x ,t \geq 0$.

\begin{theorem}\label{thmsub}
The process $\mathcal{A}_{\alpha,\gamma}$ defined by \eqref{rel} is almost surely starting from $0$ and non-decreasing. Its transition density
$h_{\alpha,\gamma}(x,t):=\mathbb{P}\left( \mathcal{A}_{\alpha,\gamma}(t)\in dx\right)$ has the following LT, w.r.t. $t,$
\begin{equation}\label{lth}
    \Tilde{h}_{\alpha, \gamma}(x,\eta)=(\alpha+\gamma)x^{\alpha+\gamma-1}\sum_{l=0}^{\infty}(l+1)(-\eta)^lx^{l(\alpha+\gamma)}c_{l+1},
\end{equation}
where \begin{equation}
c_n:=\prod_{j=0}^{n-1}\frac{\Gamma(j(\gamma+\alpha)+\gamma+1)}{\Gamma((j+1)(\gamma+\alpha)+1)}
 =
\frac{1}{[\mathsf{n}!]_\alpha^{\alpha+\gamma}} , \qquad n=1,2,... \label{cn}
\end{equation}
\end{theorem}

\begin{proof}
As a consequence of \eqref{rel}, the process $\mathcal{A}_{\alpha,\gamma}$ starts a.s. from $0$, since it is defined as the first time over a certain level for $\mathcal{Z}_{\alpha,\gamma}$ and $\mathcal{Z}_{\alpha,\gamma}(0)=\left. t^{\alpha+\gamma}Z_{\alpha,\gamma} \right \vert_{t=0}=0$ a.s. Moreover, it is a.s. non-decreasing since it is the inverse of an a.s. non-decreasing process.

Finally, if we denote the transition density of $\mathcal{Z}_{\alpha,\gamma}$ by $f_{\alpha,\gamma}(x,t):=\mathbb{P}(\mathcal{Z}_{\alpha,\gamma}(t) \in dx)$, then the LT of the transition density $h_{\alpha,\gamma}(z,t)$, w.r.t. $t$, must satisfy the following equation
\begin{equation*}
    \int_0^x \Tilde{h}_{\alpha, \gamma}(z,\eta)dz=\frac{1}{\eta}-\frac{\Tilde{f}_{\alpha, \gamma}(\eta,x)}{\eta}, \qquad x \geq 0, \eta >0,
\end{equation*}
where $\Tilde{f}_{\alpha, \gamma}(\eta,x)=\int_0^\infty e^{-\eta z}f_{\alpha, \gamma}(z,x)dz$ and
$\Tilde{h}_{\alpha, \gamma}(z,\eta)=\int_0^\infty e^{-\eta t}h_{\alpha, \gamma}(z,t)dt$.
Thus, in our case, it must be
\begin{eqnarray}
\Tilde{h}_{\alpha, \gamma}(x,\eta)&=&-\frac{1}{\eta}\frac{d}{dx}\operatorname{E}_{\alpha,1+\gamma/\alpha,\gamma/\alpha}\left(-\eta x^{\alpha+\gamma}\right) \label{teo2} \\
&=&(\alpha+\gamma)x^{\alpha+\gamma-1}\sum_{l=0}^{\infty}(l+1)(-\eta)^lx^{l(\alpha+\gamma)}\prod_{j=0}^l \frac{\Gamma(j(\alpha+\gamma)+\gamma+1)}{\Gamma((j+1)(\alpha+\gamma)+1)}, \notag
\end{eqnarray}
which gives \eqref{lth} with \eqref{cn}, by recalling the definitions given in eq.\eqref{def.n.beta.alpha} and eq.\eqref{def.n.beta.alpha.factorial}.
To prove the convergence of the series in \eqref{lth},
we can use Gautschi's inequality (see \cite{Gautschi}):
for $x > 0$ and $\sigma \in [0,1]$, it holds that
\begin{equation}
\label{gaut}
x^{1-\sigma} \leq \frac{\Gamma(x+1)}{\Gamma(x+\sigma)} \leq (x+1)^{1-\sigma} .
\end{equation}
Thus, for $[ \mathsf{n}]^\beta_\alpha$ in eq.\eqref{def.n.beta.alpha} we have that
\begin{equation}
\label{def.n.beta.alpha.gautschi}
\frac{1}{(1+\beta n)^\alpha} \leq \frac{1}{[ \mathsf{n}]^\beta_\alpha} \leq \frac{1}{(\beta n)^\alpha} ,
\end{equation}
for $x = \beta n$ and $\sigma = 1-\alpha$.
Then we have
\begin{equation*}
\lim_{l\to \infty} \left|\frac{(l+2)c_{l+2}}{(l+1)c_{l+1}}\right| =
\lim_{l\to \infty} \left|\frac{(l+2)[\mathsf{(l+1)!}]_\alpha^{\alpha+\gamma}}{(l+1)[\mathsf{(l+2)!}]_\alpha^{\alpha+\gamma}}\right| =
\lim_{l\to \infty} \left|\frac{1}{[\mathsf{l+2 }]_\alpha^{\alpha+\gamma}}\right| = 0 .
\end{equation*}
\end{proof}

It is easy to check that  the expression given in \eqref{lth} is non-negative, by considering \eqref{teo2}, and recalling that $\operatorname{E}_{\alpha,1+\gamma/\alpha,\gamma/\alpha}\left(-\eta x^{\alpha+\gamma}\right)$
is monotonically decreasing (see discussion at the end of Section~\ref{sub_sec_KS}).

Moreover, in the special case $\gamma=0$, eq.\eqref{lth} coincides with the LT, w.r.t. the time argument, of the transition density of an $\alpha$-stable subordinator, i.e. $\Tilde{h}_{\alpha}(x,\eta)=x^{\alpha -1}\operatorname{E}_{\alpha,\alpha}\left(-\eta x^{\alpha}\right) $, where
$\operatorname{E}_{\alpha,\beta}(z)$ is given by eq.\eqref{ML_2_param}.

Thus, for $\gamma \neq 0$, Theorem \ref{thmsub} introduces a generalization of the $\alpha$-stable subordinator, since $\mathcal{A}_{\alpha,\gamma}$ is a.s. starting from $0$, non-decreasing and its transition density reduces to that of $\mathcal{A}_{\alpha}$, for $\gamma=0$; however, only in this special case the process has independent, stationary increments and thus it is a L\'{e}vy subordinator.

\vspace{0.5cm}
We now give the following time-change representation for the interarrival times and for the $n$-th order waiting times of $\mathcal{N}_{\alpha,\gamma}$.
\begin{theorem}
    Let $U$ be an exponentially distributed r.v. with parameter $\lambda$, independent from $\left\{  \mathcal{A}_{\alpha, \gamma}(t)\right\}_{t \geq 0}$, then the following equality in distribution holds, for the interarrival times of $\mathcal{N}_{\alpha,\gamma}$,
\begin{equation}U^{(\alpha, \gamma)}\overset{d}{=}\mathcal{A}_{\alpha, \gamma}(U). \label{uff5}
\end{equation}
Moreover, the waiting time of the $n$-th arrival $T^{(\alpha, \gamma)}_n$ has density function,  for $t \geq 0$ and $n \in \mathbb{N}$,
  \begin{equation}
f_{T^{(\alpha, \gamma)}_n}(t)=f^{*(n)}_{U^{(\alpha,\gamma)}}(t)
    =\lambda  \int_{0}^\infty h^{*(n)}_{^{\alpha,\gamma}}(t,u)e^{-\lambda u}du, \label{tn0}
\end{equation}
where $h^{*(n)}_{^{\alpha,\gamma}}(\cdot,t)$ is the $n$-th order convolution of the r.v. $\mathcal{A}_{\alpha, \gamma}(t)$, for any $t \geq 0$, under the independence assumption.
\end{theorem}
\begin{proof}
 As a consequence of eq. \eqref{ff} the interarrival times are i.i.d. random variables with
\begin{equation} \label{uff3}
    \mathbb{P}(U^{(\alpha, \gamma)}>t)=\mathbb{P}(U>\mathcal{Z}_{\alpha, \gamma}(t)),
\end{equation}
where $U$ is an exponential r.v. with parameter $\lambda$.

Thus, we can write that
    \begin{equation} \label{uff4}
    \mathbb{P}(U^{(\alpha, \gamma)}>t)=\lambda \int_0^{+\infty}\mathbb{P}(\mathcal{Z}_{\alpha, \gamma}(t)<u)e^{-\lambda u}du=\lambda \int_0^{+\infty}\mathbb{P}(\mathcal{A}_{\alpha, \gamma}(u)>t)e^{-\lambda u}du,
\end{equation}
so that \eqref{uff5} and \eqref{tn0} easily follow.
\end{proof}

\begin{remark}\label{past}
    As a consequence of the previous theorem and by recalling the definition of Poisson process (with exponential interarrival times $U_j$) given in \cite{GER} and applied to the non-homogeneous fractional case in \cite{SCA}, we can give an alternative representation of our process in terms of equality of finite-dimensional distributions. Indeed, we have that
\[
\mathcal{N}(t)\overset{a.s.}{=}\sum_{n=0}^{\infty} n \mathbbm{1}_{\left\{\zeta_n \leq t < \zeta_{n+1}\right\}}, \qquad t \geq 0,
\]
where $\zeta_n:=\zeta'_{\kappa_n}$, for $\zeta_n':=\max\left\{U_1,...U_n \right\}$, $n=1,2,...$ and $\kappa_n:=\inf\left\{k \in \mathbb{N}:\zeta'_k >\zeta'_{\kappa_{n-1}} \right\}$, $n=2,3,...$, with $\kappa_1=1.$ Then, by considering the equality in distribution given in \eqref{uff5}, we have that
\begin{equation}\label{sca1}
\mathcal{N}_{\alpha,\gamma}(t)\overset{d}{=}\sum_{n=0}^{\infty} n \mathbbm{1}_{\left\{\zeta_n \leq \mathcal{Z}_{\alpha,\gamma}(t) < \zeta_{n+1}\right\}}=\sum_{n=0}^{\infty} n \mathbbm{1}_{\left\{\frac{\zeta_n}{t^{\alpha+\gamma}} \leq Z_{\alpha,\gamma} < \frac{\zeta_{n+1}}{t^{\alpha+\gamma}}\right\}}, \qquad t> 0.
\end{equation}
%while, on the other hand, by equations \eqref{defdef} and \eqref{LTZ}, we have that
%\begin{eqnarray}
%     \mathbb{P}(U^{(\alpha, \gamma)}>t)&=& \operatorname{E}_{\alpha,1+\gamma/\alpha,\gamma/\alpha}\left(-\lambda t^{\alpha+\gamma}\right) =\int_0^\infty e^{-\lambda w}f_{\alpha,\gamma}(w,t)dw \notag  \\
    % &=& \frac{1}{t^{\alpha+\gamma   }}\int_0^\infty e^{-\lambda w}f_{\alpha,\gamma}\left( \frac{w}{t^{\alpha+\gamma}},1\right)dw. \notag
%\end{eqnarray}
%By the unique characterization of the Laplace transform the two functions should coincide if the two probabilities would be equal. However, it can be checked that $\frac{1}{t^{\alpha+\gamma}}f_{\alpha,\gamma}(w/t^{\alpha+\gamma)},t)$ and $\frac{1}{t^{\alpha}}l_\alpha\left((w(\gamma+1))^{1/(\gamma+1)}/t^\alpha,1\right)$ coincide only in the special cases $\alpha=1$ or $\gamma=0$.

\end{remark}

\begin{remark}
 We compare the process $\mathcal{N}_{\alpha,\gamma}$ to the fractional non-homogeneous Poisson process introduced in \cite{SCA}, and we prove that they coincide in distribution only for $\alpha=1$.

 Let $\Lambda(t):=\int^t_0 \lambda(\tau)d\tau$ and let $\left\{\mathcal{N}_1(t)\right\}_{t \geq 0}$ be a standard Poisson process with parameter equal to $1$, then the non-homogeneous Poisson process $\mathcal{N}^I:=\left\{\mathcal{N}^I(t)\right\}_{t \geq 0}$ governed by the equation
 \begin{equation}\frac{d}{dt}p(n;t)+\lambda(t+u)(p(n;t)-p(n-1;t))=0, \qquad n=0,1,2.., t \geq 0, u \geq 0, \label{inhom}
 \end{equation}
 with initial condition $p(n;0)=\delta_0(n)$, has the following time-change representation
$\mathcal{N}^I(t)=\mathcal{N}_1(\Lambda(t))
 $ (see \cite{SCA}).
Equation \eqref{inhom} coincides with \eqref{las}, in the particular case where $\alpha=1$, for the special choice $\lambda(t)=\lambda t^{\gamma}$ and $u=0$. Moreover, we have that
 \begin{eqnarray}
   \mathbb{P}(\mathcal{N}^I(t)=0)= e^{-\Lambda(t)}&=&e^{-\frac{\lambda t^{\gamma+1}}{\gamma+1}} =1+\sum_{n=1}^\infty\prod_{k=1}^n \frac{(\lambda t^{\gamma+1})^n}{n!(\gamma+1)^n} \notag\\
&=&1+\sum_{n=1}^\infty\prod_{k=1}^n \frac{\Gamma(k(\gamma+1)}{\Gamma(k(\gamma+1)+1)}(\lambda t^{\gamma+1})^n =\operatorname{E}_{1,1+\gamma,\gamma}\left(-\lambda t^{\gamma+1}\right), \notag
    \end{eqnarray}
which coincides with $\mathbb{P}(\mathcal{N}_{\alpha,\gamma}(t)=0)\mathbb={P}(U^{(\alpha, \gamma)}>t)$ given in \eqref{defdef}, only in the special case $\alpha=1$.
Alternatively, we now consider the time-change representation  given in \cite{SCA} for the fractional non-homogeneous Poisson process, i.e. $\mathcal{N}^I_{\alpha}(t):=\mathcal{N}_1(\Lambda(\mathcal{L_\alpha}(t)))=\mathcal{N}^I(\mathcal{L_\alpha}(t))$, for $t \geq 0$, where the inverse stable subordinator is assumed to be independent of the process $\mathcal{N}^I$. Under the same choice of $\lambda(t)=\lambda t^\gamma$, we can show that the two processes  $\mathcal{N}_{\alpha,\gamma}$ and $\mathcal{N}^I_\alpha$ are not equal, even in distribution. Indeed, if we denote the interarrival times of the process $\mathcal{N}_\alpha^I$ as $U^I_\alpha$ and we recall that $l_\alpha(x,t)=\frac{1}{t^\alpha}W_{-\alpha,1-\alpha}(-x/t^\alpha)$, where $W_{\alpha,\beta}(\cdot)$ is the Wright function (see \cite{KIL}, p.56), we can write  that
\begin{eqnarray}
     \mathbb{P}(U^I_{\alpha}>t)&=&\int_0^\infty  e^{-\Lambda(z)} l_\alpha(z,t)dz=\int_0^\infty e^{-\frac{\lambda z^{\gamma+1}}{\gamma+1}}l_\alpha\left(z,t\right)dz \notag \\
&=&(\gamma+1)^{-\gamma/(\gamma+1)}\int_0^\infty e^{-\lambda w}w^{-\gamma/(\gamma+1)}l_\alpha\left((w(\gamma+1))^{1/(\gamma+1)},t\right)dw, \notag
\end{eqnarray}
which does not coincide with $\mathbb{P}(U^{(\alpha, \gamma)}>t)$ given in equation \eqref{defdef}.

\end{remark}

For a full characterization of $\mathcal{N}_{\alpha,\gamma}$ and for its covariance function,  we recall the following result on the  moments of renewal processes.
\begin{lem}\label{lemsuy}(\cite{SUY})
Let $\left\{M(t)\right\}_{t \geq 0}$ be a renewal process with interarrival time $Y$, then
\begin{equation}\notag
  \int_{0}^{+\infty}  e^{-\eta t} \mathbb{E}M(t)dt=\frac{\Tilde{f}_{Y}(\eta)}{\eta \left( 1-\Tilde{f}_{Y}(\eta)\right)}, \qquad \eta\geq 0,
\end{equation}
where $f_Y(\cdot)$ is the density function of $Y$, and
\begin{equation}
    \int_{0}^{+\infty} \int_{0}^{+\infty} e^{-\eta_1 t_1-\eta_2 t_2}\mathbb{E}\left(M(t_1) M(t_2) \right)dt_1 dt_2 = \frac{ \left[ 1-\Tilde{f}_Y(\eta_1) \Tilde{f}_Y(\eta_2)\right] \Tilde{f}_Y(\eta_1+\eta_2)}{\eta_1 \eta_2 \left[ 1-\Tilde{f}_Y(\eta_1) \right] \left[ 1-\Tilde{f}_Y(\eta_2)\right]\left[1-\Tilde{f}_Y(\eta_1+\eta_2)\right]}. \notag
\end{equation}
\end{lem}

Moreover, as a preliminary result, we derive the Laplace transform of the Kilbas-Saigo function (hereafter denoted as $\mathcal{L}_t[\operatorname{E}_{a,m,l}(-\lambda t^\nu);z]$) in a very convenient form for numerical integration. Indeed, formula \eqref{lem2} below can be used to obtain plots of $\mathcal{L}_t[\operatorname{E}_{a,m,l}(-\lambda t^\nu);z]$, by using, for example, Mathematica.
\begin{lem}\label{laplacegeta}  %%% COLOR
Let $G(z;\delta)$, for $\delta >0$ and $z \in \mathbb{C}$, be the  double Gamma function, and let
\[
\Theta(s):=\frac{\Gamma(s)\Gamma(1/\nu-s/\nu)\Gamma(1-1/\nu-s/\nu)G(\varphi-1/\nu+s/\nu;\tau)}{G(\varphi+a\tau-1/\nu+s/\nu;\tau)},
\]
where $\varphi:=(1+al)\tau$, $\tau:=1/am$ and $\nu > 0$. Then, for $\lambda>0,$ the Laplace transform of the Kilbas-Saigo function reads
\begin{equation}\label{lem2}
    \int_0^\infty e^{-zt}\operatorname{E}_{a,m,l}(-\lambda t^\nu)dt=\frac{\lambda^{-1/\nu}G(\varphi+a\tau;\tau)}{\nu G(\varphi;\tau)}\frac{1}{2\pi i}\int_{\mathcal{C}}(\lambda^{-1/\nu}z)^{-s}\Theta(s)ds.
\end{equation}
where $C$ is the contour $(c-i\infty,c+i\infty)$ with $c_0 < c < 1$ and
$c_0 = \operatorname{max}[0,1-\nu,1-\nu(\varphi + a\tau)]$.
\end{lem}
\begin{proof}
In \cite{BLPV} it is shown that the KS function can be written,
for $\operatorname{Re}z > 0$, as
\begin{equation}
\label{KS.mellin}
\operatorname{E}_{a,m,l}(-z) =  \frac{1}{2\pi i}\frac{G(\varphi+a\tau  ;\tau)}{G(\varphi;\tau)} \int_{\mathcal{C}} \frac{\Gamma(s)\Gamma(1-s)G(\varphi  -s;\tau)}{G(\varphi+a\tau -s;\tau)} z^{-s} \, ds
\end{equation}
with  $\varphi$ and
$\tau$ as given in \eqref{def.varphi} and \eqref{exp.tau}, respectively,
and we take $z^{-s} = \operatorname{e}^{-s \operatorname{Log} z}$, with $\operatorname{Log}{z}$ denoting the principal branch of the logarithm.
Indeed, it is not difficult to obtain, from eq.\eqref{KS.mellin}, the
representation of the KS function in eq.\eqref{def.KS.G.pochhammer} \cite{BLPV}. Equation \eqref{KS.mellin} can be interpreted as
\begin{equation}
\mathcal{M}_z[\operatorname{E}_{a,m,l}(-z);s] = \frac{G(\varphi+a\tau  ;\tau)}{G(\varphi;\tau)}  \frac{\Gamma(s)\Gamma(1-s)G(\varphi  -s;\tau)}{G(\varphi+a\tau -s;\tau)} ,\notag
\end{equation}
where $\mathcal{M}_z[f(z),s]$ denotes the Mellin transform of $f(z)$ with conjugate variable $s$. Using the well-known property
\begin{equation}\notag
\mathcal{M}_z[f(\lambda z^\nu);s] = \frac{1}{\nu}\lambda^{-s/\nu}F(s/\nu) ,
\end{equation}
with $\lambda > 0$ and $\nu > 0$ we have
\begin{equation}\notag
\mathcal{M}_z[\operatorname{E}_{a,m,l}(-\lambda z^\nu);s] =  \frac{G(\varphi+a\tau  ;\tau)}{\nu  G(\varphi;\tau)}  \, \frac{ \lambda^{-s/\nu}  \Gamma(s/\nu)\Gamma[1-s/\nu]G(\varphi  -s/\nu;\tau)}{G(\varphi+a\tau -s/\nu;\tau)} .
\end{equation}
The Laplace and Mellin transforms are related by
\begin{equation}\notag
\mathcal{M}_z[\mathcal{L}_t[f(t);z];s] = \Gamma(s)\mathcal{M}_z[f(z);1-s] .
\end{equation}
This property follows by changing the order of integration in $\mathcal{M}_z[\mathcal{L}_t[f(t);z];s]$.
This is justified for $f(z) = \operatorname{E}_{a,m,l}(-\lambda z^\mu)$ because of the absolute
convergence of the double integral (Fubini-Tonelli theorem) which follows
from the absolute convergence of the Mellin transform of $\operatorname{E}_{a,m,l}(-\lambda z^\nu)$ together with the
exponential damping provided by the Laplace kernel.
Thus
\begin{equation}\notag
\mathcal{M}_z[\mathcal{L}_t[\operatorname{E}_{a,m,l}(-\lambda t^\nu);z];s] =
\frac{\lambda^{-1/\nu} G(\varphi+a\tau  ;\tau)}{\nu  G(\varphi;\tau)} \lambda^{s/\nu}   \Theta(s) ,
\end{equation}
with
\begin{equation}\notag
\Theta(s) = \frac{  \Gamma(s) \Gamma(1/\nu -s/\nu)\Gamma(1-1/\nu +s/\nu)G(\varphi -1/\nu  +s/\nu;\tau)}{G(\varphi+a\tau -1/\nu + s/\nu;\tau)}
\end{equation}
Thus we obtain for the Laplace transform of $\operatorname{E}_{a,m,l}(-\lambda t^\nu)$ that
\begin{equation}
\label{laplace.mellin}
\mathcal{L}_t[\operatorname{E}_{a,m,l}(-\lambda t^\nu);z] =
\frac{\lambda^{-1/\nu} G(\varphi+a\tau  ;\tau)}{\nu  G(\varphi;\tau)} \, \frac{1}{2\pi i}
\int_{\mathcal{C}} \left(\lambda^{-1/\nu} z \right)^{-s} \Theta(s)\, ds .
\end{equation}
Due to the gamma functions in the numerator, the
function $\Theta(s)$ has poles at $s=-n$ ($n=0,1,2,\ldots$) from $\Gamma(s)$, at
$s = 1 + n\nu$ ($n=0,1,2,\ldots$) from $\Gamma(1/\nu-s/\nu)$, and at
$s = 1 - (n+1)\nu$ ($n =0,1,2,\ldots$) from $\Gamma(1-1/\nu +s/\nu)$.
Since the double gamma function is entire,
the other poles of $\Theta(s)$ come from the zeros of   $G(\varphi+a\tau-1/\nu +s/\nu;\tau)$
in the denominator, which are located
at $s = 1-\nu(m\tau + n + \varphi + a\tau)$, for $m,n = 0,1,2,\ldots$.
We will choose the contour
of integration $\mathcal{C}$ such that it separates the poles of $\Gamma(s)$,
$\Gamma(1-1/\nu +s/\nu)$ and $1/G(\varphi+a\tau-1/\nu +s/\nu;\tau)$
from the poles of  $\Gamma(1/\nu-s/\nu)$. Let
$c_0 = \operatorname{max}[0,1-\nu,1-\nu(\varphi+a\tau)]$.
Clearly $c_0 < 1$. Then $C = (c-i\infty,c+i\infty)$, where $c_0 < c < 1$.
In Appendix~B we show that this integral indeed exists.
\end{proof}

By applying the previous lemmas, we can give the renewal function and the auto-covariance of $\mathcal{N}_{\alpha, \gamma}$, at least in Laplace domain. We will consider, for the sake of brevity, $\lambda=1$.
\begin{coro}\label{corocov}

Let us denote
\begin{equation}\label{geta}
g(\eta):=\frac{1}{2i(\alpha+\gamma)\mathcal{G}_\gamma(0;1/(\alpha+\gamma)) }\int_{\mathcal{C}}\eta^{1-s}\frac{\Gamma(s) \mathcal{G}_\gamma(s-1;1/(\alpha+\gamma))}{\sin(\pi(1-s)/(\alpha+\gamma))}ds,
\end{equation}
where
$$\mathcal{G}_{\gamma}(A;\tau):=\frac{G(\tau(\gamma+1)+A\tau;\tau)}{
G(\tau+1+A\tau;\tau)
}.$$
Then
 the LT of the renewal function and of the auto-covariance function of $\mathcal{N}_{\alpha,\gamma}$ are given, respectively, by

\begin{equation*}
 \int_{0}^{+\infty} e^{-\eta t}\mathbb{E}\mathcal{N}_{\alpha, \gamma}(t)dt =\frac{1-g(\eta)}{\eta g(\eta)}, \qquad \eta >0,
\end{equation*}
and
\begin{equation*}
    \int_{0}^{+\infty} \int_{0}^{+\infty} e^{-\eta_1 t_1-\eta_2 t_2}Cov\left(\mathcal{N}_{\alpha, \gamma}(t_1), \mathcal{N}_{\alpha, \gamma}(t_2) \right)dt_1 dt_2 = \frac{ g(\eta_1)+ g(\eta_2)- g(\eta_1)g(\eta_2)-g(\eta_1+\eta_2)}{\eta_1 \eta_2g(\eta_1)g(\eta_2)g(\eta_1+\eta_2)},
\end{equation*}
for $\eta_1, \eta_2 >0.$

\end{coro}
\begin{proof}

By taking into account \eqref{defdef}, we can write that
\begin{eqnarray}
 \Tilde{f}_{U^{(\alpha,\gamma)}}(\eta)&=&
 1-\eta \int_0^\infty e^{-\eta t}\mathbb{P}(U^{(\alpha, \gamma)}>t)dt=1-\eta \int_0^\infty e^{-\eta t}\operatorname{E}_{\alpha,1+\gamma/\alpha,\gamma/\alpha}\left(- t^{\alpha+\gamma}\right)dt \notag\\
 &=&1-g(\eta) \label{geta2}
\end{eqnarray}
so that the results easily follow, by applying the reflection formula of the Gamma function in \eqref{lem2} and Lemma \ref{laplacegeta}, with $\tau=1/(\gamma+\alpha)$, $\varphi=(1+\gamma)/(\alpha+\gamma)$ and $\nu=\alpha+\gamma$, in order to derive \eqref{geta}.

\end{proof}

\begin{remark}
Note that the general case (i.e. for $\lambda \neq 1$, as in \eqref{defdef}) can be easily obtained by considering equation \eqref{geta2}. Indeed, in this case, we have that $\Tilde{f}_{U^{(\alpha,\gamma)}}(\eta)=1-g(\eta \lambda^{-1/(\alpha+\gamma)})$.

As special case of the previous corollary, if we put $\gamma=0$, we obtain from Def. \ref{defren} that $\mathcal{N}_{\alpha, 0}=:\mathcal{N}_{\alpha}$ coincides with the fractional Poisson process introduced by \cite{LAS} and \cite{MAI}. Indeed, its interarrival time survival function satisfies \eqref{uff}, with $\gamma=0$ (see \cite{BEG2009}).

Accordingly, the equality in distribution given in equation \eqref{uff5} generalizes, to the case $\gamma \neq  0$, the well-known time-change representation of the fractional Poisson process' interarrival times, i.e. \begin{equation}U^{(\alpha)}\overset{d}{=}\mathcal{A}_{\alpha}(U) \notag
\end{equation}
where $U$ is an exponential r.v. and $\mathcal{A}_\alpha(t)$, $t \geq 0$, is an independent $\alpha$-stable subordinator  (see \cite{MEE}). In this particular case, the mean interarrival time is infinite, as Theorem \ref{rem2.3} shows.

Finally, the full characterization of $\mathcal{N}_{\alpha, \gamma}$ given in Corollary \ref{corocov} provides, for $\gamma=0$, the LT of the covariance function of the fractional Poisson process:
\begin{equation}
    \int_{0}^{+\infty} \int_{0}^{+\infty} e^{-\eta_1 t_1-\eta_2 t_2}Cov\left(\mathcal{N}_{\alpha}(t_1), \mathcal{N}_{\alpha}(t_2) \right)dt_1 dt_2 = \frac{\eta_1^\alpha+\eta_2^\alpha+\eta_1^\alpha\eta_2^\alpha-(\eta_1+\eta_2)^\alpha}{\eta_1^{\alpha+1}\eta_2^{\alpha+1}(\eta_1+\eta_2)^\alpha}, \label{FPPcov}
\end{equation}
since, in this case,
$g(\eta)=\eta^\alpha/(\eta^\alpha+1)$. It can be checked, with some algebraic calculations, that formula \eqref{FPPcov} agrees with the result given in \cite{LEO}.

\end{remark}

\paragraph{Alternative (non-renewal) approach}

We now move from an alternative starting point, i.e. we introduce the stretched operator $\operatorname{\mathcal{D}}^{(\alpha,\gamma)}_{t}$ in the difference-differential equation governing the Poisson process; see equation \eqref{fde.3} below.  The latter is equivalent to equation \eqref{las} and thus the following result provides a time change-representation for the fractional non-homogeneous model studied by Laskin in \cite{LAS}.
\begin{theorem}\label{thmlas}
   Let $\alpha \in (0,1)$, $\alpha + \gamma > 0$ and let $\left\{p_{\alpha,\gamma}^L (n,t)\right\}_{n \geq 0}$ be the solution to equation
    \begin{equation}
\label{fde.3}
\operatorname{\mathcal{D}}^{(\alpha,\gamma)}_t p(n;t) + \lambda (p(n;t)-p(n-1;t)) = 0, \qquad n\in \mathbb{N}_0,t \geq 0,
\end{equation}
with $p_{\alpha, \gamma}(n;0)=\delta_0(n)$, $n=0,1,...$, and let $\mathcal{N}^L:=\left\{\mathcal{N}^L(t) \right\}_{t \geq 0}$ be the process such that $\mathbb{P}(\mathcal{N}^L(t)=n)=p_{\alpha,\gamma}^L(n;t)$, for any $n \geq 0$ and $t \geq 0$.

Then
the following equality in distribution holds
\begin{equation}\label{fdd}
\left\{\mathcal{N}^L_{\alpha, \gamma}(t)\right\}_{t \geq 0}\overset{d}{=} \left\{\mathcal{N}(\mathcal{Z}_{\alpha, \gamma}(t))\right\}_{t \geq 0},
\end{equation}
under the assumption that
$\left\{\mathcal{N}(t)\right\}_{t \geq 0}$ and $\left\{\mathcal{Z}_{\alpha, \gamma}(t)\right\}_{t \geq 0}$ are independent.
\end{theorem}
\begin{proof}
Let $G(u,t):=\sum_{n=0}^\infty u^n p(n;t)=e^{-\lambda t(u-1)},$ $t \geq 0$, be the well-known probability generating function of the Poisson process $\mathcal{N}$, with parameter $\lambda$,  then, as a consequence of \eqref{fdd}, we can write that
\begin{eqnarray}\label{pgf}
G_{\mathcal{N}(\mathcal{{Z}}_{\alpha, \gamma})}(u,t)&:=&\int_0^\infty e^{-\lambda z(1-u)}f_{\alpha, \gamma}(z,t)dz
=\operatorname{E}_{\alpha,1+\gamma/\alpha,\gamma/\alpha}\left(-\lambda (1-u)t^{\alpha+\gamma}\right),
\end{eqnarray}
for $t \geq 0$, $|u| \leq 1$, where again $f_{\alpha,\gamma}(x,t):=\mathbb{P}(\mathcal{Z}_{\alpha,\gamma}(t) \in dx)$.
It follows from the preliminary result at the beginning of this section that the function in \eqref{pgf} satisfies equation \eqref{fde.1} with $\kappa=\lambda(1-u)$ and initial condition $G_{\mathcal{N}(\mathcal{{Z}}_{\alpha, \gamma})}(u,0)=1$, i.e.
\[
\operatorname{\mathcal{D}}^{(\alpha,\gamma)}_t G_{\mathcal{N}(\mathcal{{Z}}_{\alpha, \gamma})}(u,t)+\lambda G_{\mathcal{N}(\mathcal{{Z}}_{\alpha, \gamma})}(u,t)-\lambda u G_{\mathcal{N}(\mathcal{{Z}}_{\alpha, \gamma})}(u,t)=0,
\]
which coincides with equation \eqref{fde.3}, after multiplying the latter by $u^n$, adding over $n=0,1,...$ and considering that $p(-1,t)=0$.
\end{proof}
    This result generalizes, to the case $\gamma \neq  0$, the well-known time-change representation of the fractional Poisson process, i.e. $\mathcal{N}_{\alpha}:=\left\{\mathcal{N}_{\alpha}(t)\right\}_{t \geq 0}$, as $\mathcal{N}_\alpha(t)=\mathcal{N}(\mathcal{L}_\alpha(t))$, $t \geq 0$, where $\mathcal{L}_\alpha(t)$ is the inverse of an independent $\alpha$-stable subordinator $\mathcal{A}_\alpha(t)$, $t \geq 0$ (see \cite{MEE}).

As a consequence of \eqref{fdd}, we can evaluate the mean and variance of $\mathcal{N}^L_{\alpha,\gamma}$, by recalling the following general result given in \cite{LEO}, for a time-changed L\'{e}vy process.
\begin{lem}\label{lemLEO}(\cite{LEO})
Let $X:=\left\{X(t)\right\}_{t \geq 0}$ be a homogeneous L\'{e}vy process, starting from $0$ and let $Y:=\left\{ Y(t)\right\}_{t \geq 0}$ be an independent non-decreasing process. If $\mu_X:=\mathbb{E}X(1)< \infty$, $\sigma^2_X:=Var X(1)< \infty$ and $\mathbb{E}Y^r(t)< \infty$, $r=1,2,...,$ and $t \geq 0$, then
\begin{eqnarray}
&&\mathbb{E}X(Y(t))=\mu_X \mathbb{E}Y(t) \notag \\
     && Cov\left[ X(Y(t)),X(Y(s))\right]=\sigma^2_X \mathbb{E}Y(s \wedge t))+\mu_X^2 Cov\left[ Y(s),Y(t)\right]. \label{covleo}
\end{eqnarray}
\end{lem}

%, at least in the Laplace domain. For simplicity and without loss of generality, let us assume hereafter that $\lambda=1$.

\begin{coro}
For the process $\mathcal{N}^L_{\alpha,\gamma}$ we have that
\begin{equation}\label{rf}
\mathbb{E}\mathcal{N}^L_{\alpha, \gamma}(t)=\frac{\lambda\Gamma(\gamma+1)}{
\Gamma(\alpha+\gamma+1)}t^{\alpha+\gamma}
\end{equation}
and

\begin{equation}\label{varn}
Var \mathcal{N}^L_{\alpha, \gamma}(t) = \frac{\lambda \Gamma(\gamma+1)}{\Gamma(\alpha+\gamma+1)} t^{\alpha+\gamma}+ \frac{\lambda^2 \Gamma^2(\gamma+1)}{\Gamma^2(\alpha+\gamma+1)}A_{\alpha,\gamma} t^{2(\alpha+\gamma)},
\end{equation}
%where \[A_{\alpha,\gamma} :=\prod_{n=0}^\infty \frac{(\gamma+\alpha)(1-\alpha)}{(\alpha+\gamma+n)(\alpha+2\gamma+n+1)}.\]
where \[A_{\alpha,\gamma} :=\prod_{n=0}^\infty \frac{\alpha(n+1)+\gamma+(2\gamma+n+1)(\alpha+\gamma +n)}{(\alpha+\gamma+n)(\alpha+2\gamma+n+1)}-1.\]

\end{coro}
\begin{proof}
The expected value follows by simple calculations, by recalling that $\mathcal{Z}_{\alpha,\gamma}(t)\overset{d}{=}t^{\alpha+\gamma}Z_{\alpha,\gamma}$, with $Z_{\alpha,\gamma}$ as in \eqref{beta}, and that for a Beta r.v.'s
$B(a,b)$ it is $\mathbb{E}B(a,b)=a/(a+b)$. As far as the variance is concerned, by \eqref{covleo}, for $s=t$, by  the independence of the Beta r.v.'s and recalling that
$\mathbb{E}B^2\left(a,b\right)= \frac{a^2}{(a+b)^2}\left[\frac{b}{a(a+b+1)} +1\right]$, we can write
\begin{eqnarray}
\mathbb{E}\left[\mathcal{Z}^2_{\alpha, \gamma}(t)  \right]&=&t^{2(\alpha+\gamma)}\mathbb{E}Z_{\alpha,\gamma}^2 \notag \\
&=& t^{2(\alpha+\gamma)}\frac{ \Gamma^2(\gamma+1)}{\Gamma^2(\alpha+\gamma+1)}\prod_{n=0}^\infty \frac{(\gamma+n+1)^2}{(\gamma+\alpha+n)^2}\mathbb{E}B^2
\left(\frac{\alpha+\gamma+n}{\alpha +\gamma},\frac{1-\alpha}{\alpha+\gamma}\right), \notag
\end{eqnarray}
so that \eqref{varn} follows after some algebraic calculations.
\end{proof}

\begin{remark}
    Formula \eqref{rf} coincides, in the case $\alpha+\gamma<1$, with the result obtained in \cite{LAS}, by a different method. Moreover, in the special case $\gamma=0$, we can check that the variance given in \eqref{varn} coincides with the variance of the fractional Poisson process obtained in \cite{LEO}, i.e.
   \begin{equation*}
Var \mathcal{N}_{\alpha}(t) = \frac{\lambda t^{\alpha} }{\Gamma(\alpha+1)} + \frac{\lambda^2 t^{2\alpha}}{\Gamma^2(\alpha+1)}\left[ \frac{\Gamma(\alpha) \Gamma(\alpha+1)}{\Gamma(2\alpha)} -1 \right],
\end{equation*}
    by recalling the well-known infinite product representation of Gamma functions' ratio \[\frac{\Gamma(b_1) \Gamma(b_2)\cdots\Gamma(b_m)}{\Gamma(a_1) \Gamma(a_2)\cdots\Gamma(a_m)}=\prod_{n=0}^\infty\frac{(a_1+n)(a_2+n) \cdots (a_m+n)}{(b_1+n)(b_2+n) \cdots (b_m+n)}
    \]
   for $m \in \mathbb{N}$, $a_n \in \mathbb{R}$ and $b_n\neq 0, -1, -2,...$, for any $n \geq 0$ (see, for example, \cite{CHA}).

   Finally, we note that, as an easy consequence of equation \eqref{covleo}, any Poisson process time-changed by a non-decreasing process is overdispersed since
   \[Var X(Y(t))=\sigma^2_X \mathbb{E}Y(t)+\mu_X^2 Var Y(t)=\mu_X \mathbb{E}Y(t)+\mu_X^2 Var Y(t)=\mathbb{E}X(Y(t))+\mu_X^2 Var Y(t).\]
Thus, also in this case, we have that $Var \mathcal{N}^L_{\alpha, \gamma}(t) > \mathbb{E} \mathcal{N}^L_{\alpha, \gamma}(t)$, for any $t > 0.$

\end{remark}

\section{A second-order equation involving $\mathcal{D}^{(\alpha,\gamma)}_t$}

We now consider the following fractional second-order equation
\begin{equation}
\label{eq.2}
\left(\operatorname{\mathcal{D}}^{(\alpha,\gamma)}_t\right)^2 f(t) + a\operatorname{\mathcal{D}}^{(\alpha,\gamma)}_t f(t) +
b f(t) = 0 , \qquad t > 0,
\end{equation}
where $\alpha \in (0,1)$, $\alpha + \gamma > 0$ and $a,b \in \mathbb{R}^+$.
It generalizes the equation satisfied by the characteristic function of the fractional telegraph process, which is obtained in the special case $\gamma=0$ (see \cite{ORS,Vieira}, where non-sequential versions of fractional telegraph equation have been studied). Here, we are interested in \eqref{eq.2} as a second-order renewal-type equation, governing the survival probability of interarrival times, under particular choices of $a$ and $b$.

\paragraph{Preliminary results} We will seek solutions to equation (\ref{eq.2}) using the form described in (\ref{eq.3}).

\begin{theorem}\label{thm3}
  Let $\alpha \in (0,1)$, $\alpha + \gamma > 0$ and let \begin{equation}\label{eta1eta2}
  \eta_1:= \frac{a}{2}-\sqrt{\frac{a^2}{4}-b}  \qquad \text{and} \qquad  \eta_2:= \frac{a}{2}+ \sqrt{\frac{a^2}{4}-b}.
  \end{equation}
  Then, for $b \neq (a/2)^2$, the solution to equation \eqref{eq.2} is given by
\begin{equation}
f(t) = K_1 \, \operatorname{E}_{\alpha,1+\gamma/\alpha,\gamma/\alpha}\left(-\eta_1 t^{\alpha+\gamma}\right) +
K_2 \, \operatorname{E}_{\alpha,1+\gamma/\alpha,\gamma/\alpha}\left(- \eta_2 t^{\alpha+\gamma}\right) , \label{sol}
\end{equation}
for $K_1,K_2$ real constants if $a^2/4 > b$ or complex constants with $K_1 = K_2^\ast$ if $a^2/4<b$, while, for $b=(a/2)^2$, it is instead equal to
\begin{equation}\label{coro2}
f(t)=K^\prime_1\operatorname{E}_{\alpha,1+\gamma/\alpha,\gamma/\alpha}\left( (-a/2) t^{\alpha+\gamma}\right)+K_2^\prime t^{\alpha+\gamma}
\operatorname{E}^{(1)}_{\alpha,1+\gamma/\alpha,\gamma/\alpha}\left(-(a/2)t^{\alpha+\gamma}\right),
\end{equation}
where we used the notation in eq.\eqref{eq.4.a} and  $K^\prime_1,K^\prime_2$ are real constants.
 \end{theorem}
\begin{proof}
We start by considering the case $b \neq (a/2)^2$ (i.e. for $\eta_1 \neq \eta_2$): the term $\operatorname{\mathcal{D}}^{(\alpha,\gamma)}_tf(t)$
is given by eq.(\ref{der.order.1}) and
 \begin{equation*}
\left(\operatorname{\mathcal{D}}^{(\alpha,\gamma)}_t\right)^2 f(t) = \sum_{n=0}^\infty
f_{n+2}[   (\mathsf{n+2}) \times (\mathsf{n+1}) ]^{\alpha+\gamma}_\alpha t^{(\alpha+\gamma)n} ,
\end{equation*}
which in eq.(\ref{eq.2}) gives the recurrence relation
\begin{equation*}
f_{n+2} = \frac{(-a)}{[\mathsf{n+2}]^{\alpha+\gamma}_\alpha} f_{n+1} +
\frac{(-b)}{[   (\mathsf{n+2})\times (\mathsf{n+1}) ]^{\alpha+\gamma}_\alpha} f_{n} .
\end{equation*}
Using this recurrence relation we obtain
\begin{equation}
\label{gen.coeff}
f_{n} = \frac{\mathcal{U}_{n}(-a,-b)}{[ \mathsf{n!}]^{\alpha+\gamma}_\alpha} [ \mathsf{1}]^{\alpha+\gamma}_\alpha f_1 +
\frac{(-b)\mathcal{U}_{n-1}(-a,-b)}{[ \mathsf{n}!]^{\alpha+\gamma}_\alpha}  f_0 , \qquad n = 2,3,\ldots
\end{equation}
where $\mathcal{U}_{n}(-a,-b)$  is given by
\begin{equation*}
\mathcal{U}_{n}(-a,-b) = \sum_{j=0}^{\lfloor (n-1)/2\rfloor} \binom{n-1-j}{j} (-a)^{n-1-2j}
(-b)^j
\end{equation*}
where $\lfloor \cdot \rfloor$ denotes the floor function and $n = 1,2,3,\ldots$.
Note that eq.(\ref{gen.coeff}) also holds for $n=1$ if we define $\mathcal{U}_0(-a,-b) = 0$.
Thus we can identify $\mathcal{U}_n(-a,-b)$ with the \textit{bivariate
Fibonacci polynomials} $u_n(x,y)$ defined in eq.\eqref{bivariate.fib}.

The solution $f(\cdot)$ can be written therefore as
\begin{equation}
\label{sol.aux}
\begin{split}
f(t) = & f_0\left[ 1 + \sum_{n=2}^\infty F_{n-1}(-a/\sqrt{-b}) \frac{(\sqrt{-b}t^{\alpha+\gamma})^n }{[ \mathsf{n!}]^{\alpha+\gamma}_\alpha} \right] + \\
& +\frac{[\mathsf{1}]^{\alpha+\gamma}_\alpha}{\sqrt{-b}} f_1
\left[ \frac{\sqrt{-b}t^{\alpha+\gamma}}{[\mathsf{1}]^{\alpha+\gamma}_\alpha}
+ \sum_{n=2}^\infty F_{n}(-a/\sqrt{-b}) \frac{(\sqrt{-b}t^{\alpha+\gamma})^n}{[\mathsf{n!}]^{\alpha+\gamma}_\alpha} \right] .
\end{split}
\end{equation}
where we used eq.\eqref{fib.aux}.
This expression can be simplified using some results from the Fibonacci polynomials. Using eq.(\ref{fibo.1}) for $F_{n-1}(-a/\sqrt{-b})$ in eq.(\ref{sol.aux}) and defining $K_1$ and $K_2$ as
\begin{equation}\label{split}
\begin{split}
f_0 & = K_1 + K_2 , \\
  \frac{[ \mathsf{1}]^{\alpha+\gamma}_\alpha}{\sqrt{-b}} f_1 & =   \mu(-a/\sqrt{-b})  K_1 + \nu(-a/\sqrt{-b})  K_2 ,
\end{split}
\end{equation}
where $\mu$ and $\nu$ are defined in eq.\eqref{mu.nu},
we obtain, after some simplifications,
\begin{equation*}
f(t) = K_1 \sum_{n=0}^\infty [\mu(-a/\sqrt{-b})]^n \frac{(\sqrt{-b}t^{\alpha+\gamma})^n}{[ \mathsf{n!}]^{\alpha+\gamma}_\alpha} +
K_2 \sum_{n=0}^\infty [\nu(-a/\sqrt{-b})]^n \frac{(\sqrt{-b}t^{\alpha+\gamma})^n}{[ \mathsf{n!}]^{\alpha+\gamma}_\alpha} ,
\end{equation*}
where we recall the definition $[ \mathsf{0!}]^{\alpha+\gamma}_\alpha = 1$.
Furthermore, from eq.(\ref{mu.nu}) we have
\begin{equation*}
 \mu(-a/\sqrt{-b}) \sqrt{-b} = \frac{-a}{2} + \Omega  , \qquad
  \nu(-a/\sqrt{-b}) \sqrt{-b} = \frac{-a}{2} - \Omega   ,
\end{equation*}
where we define
\begin{equation}
\Omega := \sqrt{\left(\frac{-a}{2}\right)^2 - b}. \label{omega}
\end{equation}
Then
\begin{equation*}
f(t) = K_1 \sum_{n=0}^\infty  \frac{[(-a/2 + \Omega)t^{\alpha+\gamma}]^n}{[ \mathsf{n!}]^{\alpha+\gamma}_\alpha} +
K_2 \sum_{n=0}^\infty \frac{[(-a/2-\Omega)t^{\alpha+\gamma}]^n}{[ \mathsf{n!}]^{\alpha+\gamma}_\alpha} ,
\end{equation*}
Note that these two solutions are of the same form of eq.(\ref{sol.fde1}), so that formula \eqref{sol} follows.

We must study separately the case where $b=(a/2)^2$, so that $\Omega=0$, by solving, after factorization of the left hand side of eq.\eqref{eq.2}, the following system of equations:

\begin{equation*}
    \begin{cases}
{\displaystyle \operatorname{\mathcal{D}}^{(\alpha,\gamma)}_t g(t) +
(a/2) g(t)=0}  \; &  \\[1ex]
\operatorname{\mathcal{D}}^{(\alpha,\gamma)}_t f(t) +
(a/2) f(t)=g(t).&
\end{cases}
\end{equation*}
From the first equation we have that $g(t)=g_0\operatorname{E}_{\alpha,1+\gamma/\alpha,\gamma/\alpha}\left( -(a/2) t^{\alpha+\gamma}\right)$, which, inserted into the second one, gives
\begin{equation*}
\sum_{n=0}^{\infty} f_{n+1}[ \mathsf{n+1}]^{\alpha+\gamma}_\alpha t^{(\alpha+\gamma)n} + \frac{a}{2}\sum_{n=0}^{\infty} f_n t^{(\alpha+\gamma)n}= g_0 \sum_{n=0}^{\infty} \frac{(-a/2)^nt^{(\alpha+\gamma)n}}{[\mathsf{n!}]^{\alpha+\gamma}_\alpha},
\end{equation*}
by recalling \eqref{der.order.1} and \eqref{derdernow}. Thus we have that
\begin{eqnarray}
f_{n+1}&=& -\frac{a}{2}\frac{f_n}{[{\mathsf{n+1}]^{\alpha+\gamma}_\alpha}}+ \frac{(-a/2)^n g_0}{[{\mathsf{(n+1)!}]^{\alpha+\gamma}_\alpha}}
%\notag \\ &=&
=\left(-\frac{a}{2}\right)^{n+1}\frac{f_0}{[{\mathsf{(n+1)!}]^{\alpha+\gamma}_\alpha}}+ \left(-\frac{a}{2}\right)^{n}\frac{ g_0 (n+1)}{[{\mathsf{(n+1)!}]^{\alpha+\gamma}_\alpha}}, \notag
\end{eqnarray}
for $ n=0,1,2,...$, and
\begin{eqnarray}
f(t)&=&f_0+\sum_{n=0}^{\infty} f_{n+1} t^{(\alpha+\gamma)(n+1)} \notag \\
&=&f_0+f_0\sum_{n=0}^{\infty}\left(\frac{a}{2}\right)^{n+1}\frac{t^{(\alpha+\gamma)(n+1)}}{[{\mathsf{(n+1)!}]^{\alpha+\gamma}_\alpha}}+ \sum_{n=0}^{\infty} \left(-\frac{a}{2}\right)^{n}\frac{ g_0 (n+1)t^{(\alpha+\gamma)(n+1)}}{[{\mathsf{(n+1)!}]^{\alpha+\gamma}_\alpha}} \notag\\
&=& f_0 \operatorname{E}_{\alpha,1+\gamma/\alpha,\gamma/\alpha}\left( (-a/2) t^{\alpha+\gamma}\right)+g_0 t^{\alpha+\gamma}\sum_{n=0}^{\infty} \left(-\frac{a}{2}\right)^{n}\frac{ (n+1)t^{(\alpha+\gamma)n}}{[{\mathsf{(n+1)!}]^{\alpha+\gamma}_\alpha}}, \notag
\end{eqnarray}
which, by recalling eq.\eqref{eq.4.a}, coincides with \eqref{coro2}.
\end{proof}

\paragraph{Particular case}
When $\gamma = 0$  we have
\begin{equation*}
f(t) = K_1 \, \operatorname{E}_{\alpha}\left((-a/2+\Omega)t^{\alpha}\right) +
K_2 \, \operatorname{E}_{\alpha}\left((-a/2-\Omega)t^{\alpha}\right) ,
\end{equation*}
where $\operatorname{E}_\alpha(\cdot)$ is the Mittag-Leffler function, which coincides with equation (2.7) in \cite{ORS}.

\subsection{Second-order renewal models}

We now define two alternative renewal processes, by means of the second-order equation \eqref{eq.2}, by  considering the two possible cases analyzed in Theorem \ref{thm3}, i.e. for $b=(a/2)^2$ and $b \neq (a/2)^2$.

\subsubsection{Case $b =(a/2)^2$}
Let us consider $a=2\lambda$ and $b=\lambda^2$, for $\lambda >0$, which corresponds to the second case analyzed in Theorem \ref{thm3}.

\begin{definition}\label{defren2}
    Let $\alpha \in (0,1)$, $\alpha + \gamma > 0$ and let $\overline{\mathcal{N}}_{\alpha, \gamma}:=\left\{\overline{\mathcal{N}}_{\alpha,\gamma}(t) \right\}_{t \geq 0}$ be the renewal process whose  interarrival times have survival probability $\mathbb{P}(\overline{U}^{(\alpha, \gamma)}>t)$ satisfying
    \begin{equation}
\label{hyprel}
\left(\operatorname{\mathcal{D}}^{(\alpha,\gamma)}_t\right)^2 f(t) +2\lambda
\operatorname{\mathcal{D}}^{(\alpha,\gamma)}_t f(t) +
\lambda^2 f(t) = 0 ,
\end{equation}
for $t \geq 0$, $\lambda >0,$ and under the conditions $f(0)=1$ and $\left.\operatorname{\mathcal{D}}^{(\alpha,\gamma)}_tf(t)\right\vert_{t=0}=0$.
%for $\alpha +\gamma<1$.

   \end{definition}
\begin{coro}
The survival probability of the interarrival times $\overline{\mathcal{N}}_{\alpha, \gamma}$ is given by
\begin{equation}\label{survsec}
\mathbb{P}(\overline{U}^{(\alpha,\gamma)}>t)=\operatorname{E}_{\alpha,1+\gamma/\alpha,\gamma/\alpha}\left(-\lambda t^{\alpha+\gamma}\right)-\sum_{j=1}^\infty c_j j (-\lambda t^{\alpha+\gamma})^j, \qquad t \geq 0.
\end{equation}

\end{coro}
\begin{proof}
We can derive \eqref{survsec} directly from equation \eqref{coro2} (with $a=2\lambda$ and $b=\lambda^2$),
by differentiating and verifying that \eqref{survsec} satisfies eq. \eqref{hyprel} under the conditions $f(0)=1$ and $\left.\left(\operatorname{\mathcal{D}}^{(\alpha,\gamma)}_t\right)f(t)\right\vert_{t=0}=0$. Indeed the first initial condition implies that $K'_1=1$, while from the second one, we get $K'_2=\lambda$, since
%(i.e. it satisfies equation \eqref{eq.2}, for $f_0=1$ and $f_1=K_1 \eta_1+K_2 \eta_2=-\lambda$)???????????}:
\begin{eqnarray}
&&\left.\left(\operatorname{\mathcal{D}}^{(\alpha,\gamma)}_t\right)\mathbb{P}(\overline{U}^{(\alpha,\gamma)}>t)\right\vert_{t=0} \notag \\
&=&\left.\left(\operatorname{\mathcal{D}}^{(\alpha,\gamma)}_t\right)  \operatorname{E}_{\alpha,1+\gamma/\alpha,\gamma/\alpha}\left(-\lambda t^{\alpha+\gamma}\right)\right\vert_{t=0}+K'_2\left(\operatorname{\mathcal{D}}^{(\alpha,\gamma)}_t\right) \left[t^{\alpha+\gamma}\left.\sum_{j=1}^\infty c_j j \left(-\lambda  t^{(\alpha+\gamma)}\right)^{j-1}\right]\right\vert_{t=0} \notag \\
&=&-\left.\lambda\operatorname{E}_{\alpha,1+\gamma/\alpha,\gamma/\alpha}\left(-\lambda t^{\alpha+\gamma}\right)\right\vert_{t=0}-\frac{K'_2}{\lambda}\left.\sum_{j=1}^\infty c_j j (-\lambda )^j \left(\operatorname{\mathcal{D}}^{(\alpha,\gamma)}_t\right)t^{(\alpha+\gamma)j}\right\vert_{t=0} \notag \\
 &=& -\lambda- \frac{K'_2}{\lambda}\sum_{j=1}^\infty c_j j (-\left. \lambda )^j t^{(\alpha+\gamma)(j-1)} \frac{\Gamma((\alpha+\gamma)j+1)}{\Gamma((\alpha+\gamma)j-\alpha+1)}\right\vert_{t=0} =-\lambda+K'_2=0,  \notag
\end{eqnarray}
by considering \eqref{KS} and \eqref{dag} for the first and second term, respectively.

The convergence of the series in \eqref{survsec} can be proved by considering that \begin{equation*}
\lim_{l\to \infty} \left|\frac{l\thinspace c_{l}}{(l-1)c_{l-1}}\right| =
\lim_{l\to \infty} \left|\frac{\Gamma((\gamma+\alpha)(l-1)+\gamma+1)}{(\alpha+\gamma)(l-1)\Gamma((\gamma+\alpha)l)}\right| =
0 ,
\end{equation*}
and applying Gautschi's inequality \eqref{gaut}, with $\sigma=\alpha$ and $x=(\gamma+\alpha)l-\alpha$.
\end{proof}

\paragraph{Stochastic representation}
 Analogously to the representation of the first renewal model, by recalling Remark \ref{past}, the following equality in distribution holds for the second-order renewal process
\begin{equation}
\overline{\mathcal{N}}_{\alpha,\gamma}(t)\overset{d}{=}\sum_{n=0}^{\infty} n \mathbbm{1}_{\left\{\overline{\zeta}_n\leq t < \overline{\zeta}_{n+1}\right\}}, \qquad t \geq 0,
\end{equation}
 where $\overline{\zeta}_n:=\overline{\zeta}'_{\kappa_n}$, for $\overline{\zeta}_n':=\max\left\{\overline{U}^{(\alpha,\gamma)}_1,...,\overline{U}^{(\alpha,\gamma)}_n \right\}$, $n=1,2,...$ and $\kappa_n:=\inf\left\{k \in \mathbb{N}:\overline{\zeta}'_k >\overline{\zeta}'_{\kappa_{n-1}} \right\}$, $n=2,3,...$, with $\kappa_1=1.$

\paragraph{Alternative (non-renewal) approach}
On the other hand, we can generalize the process $\mathcal{N}_{\alpha,\gamma}^L$ studied in the previous section (see Theorem \ref{thmlas}), by introducing the second order term in equation \eqref{fde.3}. As we will see, this generalized process shares with  $\overline{\mathcal{N}}_{\alpha,\gamma}$ the waiting time of the first event and its probability generating function satisfies equation \eqref{hyprel}; however the two processes are not equal in distribution, as it occurred in the first-order case.

\begin{theorem}\label{thmpgf}
Let $\alpha \in (0,1)$, $\alpha + \gamma > 0$, then the solution to the following equation
\begin{equation}
\label{fde.4}
(\operatorname{\mathcal{D}}^{(\alpha,\gamma)}_t)^2 p(n;t)+2\lambda
\mathcal{D}^{(\alpha,\gamma)}_tp(n;t) + \lambda^2 (p(n;t)-p(n-1;t)) = 0,
\end{equation}
with initial conditions
\begin{equation*}
   p(n;0)=\delta_0(n), \qquad \left.\operatorname{\mathcal{D}}^{(\alpha,\gamma)}_t p(n;t)\right\vert_{t=0}=0, \qquad n=0,1,...
%    \begin{cases}
%-\lambda, \; & n=0 %, \\[1ex]
%0 , & n=1,2,...,
%\end{cases}
\end{equation*}
reads
\begin{eqnarray}\overline{p}^L_{\alpha, \gamma}(n;t)&=&\frac{\left(\lambda t^{\alpha + \gamma}\right)^{2n}}{(2n)!}\operatorname{E}_{\alpha,1+\gamma/\alpha,\gamma/\alpha}^{(2n)}\left(-\lambda t^{\alpha+\gamma}\right)+\frac{\left(\lambda t^{\alpha + \gamma}\right)^{2n+1}}{(2n+1)!}\operatorname{E}_{\alpha,1+\gamma/\alpha,\gamma/\alpha}^{(2n+1)}\left(-\lambda t^{\alpha+\gamma}\right)
\notag \\
&=&p^L_{\alpha, \gamma}(2n;t)+p^L_{\alpha, \gamma}(2n+1;t) \label{solodd}
\end{eqnarray}
for $n \in \mathbb{N}_0$, $t \geq 0$,
where $\left\{p_{\alpha, \gamma}^L(n;t)\right\}_{n \geq 0}$ is the p.m.f. of $\mathcal{N}^L_{\alpha,\gamma}$ given in \eqref{pk} (see Theorem \ref{thmlas}).
\end{theorem}
\begin{proof}
In order to solve equation \eqref{fde.4}, we first multiply both terms by $u^n$ and add over $n=0,1,...$, so that we get
\begin{equation*}(\operatorname{\mathcal{D}}^{(\alpha,\gamma)}_t)^2 \sum_{n=0}^\infty u^n p(n;t)+2\lambda
\mathcal{D}^{(\alpha,\gamma)}_t \sum_{n=0}^\infty u^n p(n;t) + \lambda^2 \sum_{n=0}^\infty u^n(p(n;t)-p(n-1;t)) = 0,
\end{equation*}
where the sums converge for $|u| \leq 1.$ Then, the generating function solves the following equation
\begin{equation}
(\operatorname{\mathcal{D}}^{(\alpha,\gamma)}_t)^2 G(u,t)+2\lambda
\mathcal{D}^{(\alpha,\gamma)}_t G(u,t) + \lambda^2 (1-u)G(u,t) = 0, \label{eqg}
\end{equation}
with $G(u,0)=1$. We now consider the solution to equation \eqref{eq.2}, for $a=2 \lambda$, $b=\lambda^2(1-u)$ and we denote, for brevity, $\eta_1=\frac{a}{2}-\Omega$ and $\eta_2=\frac{a}{2}+\Omega$. Thus, by recalling \eqref{omega}, we have that $\eta_1=\lambda(1-\sqrt{u})$
 and $\eta_2=\lambda(1+\sqrt{u})$ and the solution to \eqref{eqg} can be obtained from \eqref{sol} as follows

 \begin{equation}\label{soldue}
   G(u,t)=K_1 \operatorname{E}_{\alpha,1+\gamma/\alpha,\gamma/\alpha}\left(-\eta_1 t^{\alpha+\gamma}\right) +
K_2 \operatorname{E}_{\alpha,1+\gamma/\alpha,\gamma/\alpha}\left(-\eta_2 t^{\alpha+\gamma}\right).
 \end{equation}
Indeed, the constants $K_1:=\frac{\sqrt{u}+1}{2\sqrt{u}}$ and $K_2:=\frac{\sqrt{u}-1}{2\sqrt{u}}$, $u \neq 0$, satisfy \eqref{split} with $f_0=1$ and $f_1=K_1 \eta_1+K_2\eta_2=0$, which are equivalent to $G(u,0)=1$ and $D_t^{(\alpha,\gamma)} G(u,t)|_{t=0} = 0$ (for $u \neq 0$), respectively. The latter is satisfied since, by recalling \eqref{solodd} and \eqref{fde.3}, we have that
\begin{equation}
D_t^{(\alpha,\gamma)} G(u,t)=
\sum_{n=0}^\infty u^n D_t^{(\alpha,\gamma)} \overline{p}_{\alpha,\gamma}^L(n,t)=\lambda\sum_{n=1}^\infty u^n p_{\alpha,\gamma}^L(2n-1,t)-\lambda\sum_{n=0}^\infty u^n p_{\alpha,\gamma}^L(2n+1,t),  \label{incon2}
\end{equation}
which vanishes for $t=0$, by the initial condition.

%; therefore the solution given in \eqref{sol.aux} reduces, in this case, to
%\[
%f(t) = 1 + \sum_{n=2}^\infty F_{n-1}(-a/\sqrt{-b}) \frac{(\sqrt{-b}t^{\alpha+\gamma})^n }{[ \mathsf{n!}]^{\alpha+\gamma}_\alpha}.
%\]

We can rewrite \eqref{soldue} as
\begin{eqnarray}
   G(u,t)&=&\frac{\sqrt{u}+1}{2\sqrt{u}}\sum_{n=0}^\infty c_n (-\lambda (1-\sqrt{u})t^{\alpha + \gamma })^n + \frac{\sqrt{u}-1}{2\sqrt{u}}\sum_{n=0}^\infty c_n (-\lambda (1+\sqrt{u})t^{\alpha + \gamma })^n  \notag \\
   &=&\left[\frac{\sqrt{u}+1}{2\sqrt{u}}\sum_{n=0}^\infty  (\lambda \sqrt{u}t^{\alpha + \gamma })^n +\frac{\sqrt{u}-1}{2\sqrt{u}}\sum_{n=0}^\infty  (-\lambda \sqrt{u}t^{\alpha + \gamma })^n\right]\sum_{l=0}^\infty (-\lambda t^{\alpha +\gamma} )^l \binom{l+n}{n} c_{l+n} \notag.
 \end{eqnarray}
If we take into account that
\[
\sum_{l=0}^\infty (-\lambda t^{\alpha +\gamma} )^l\frac{(l+n)!}{ l!}c_{l+n}=\operatorname{E}_{\alpha,1+\gamma/\alpha,\gamma/\alpha}^{(n)}\left(-\lambda t^{\alpha+\gamma}\right),
\]
 we have
 \begin{eqnarray}
   && G(u,t) \notag \\
   &=&\left[\frac{\sqrt{u}+1}{2\sqrt{u}}\sum_{n=0}^\infty \frac{ (\lambda \sqrt{u}t^{\alpha + \gamma })^n}{n!} +\frac{\sqrt{u}-1}{2\sqrt{u}}\sum_{n=0}^\infty \frac{ (-\lambda \sqrt{u}t^{\alpha + \gamma })^n}{n!}\right]\operatorname{E}_{\alpha,1+\gamma/\alpha,\gamma/\alpha}^{(n)}\left(-\lambda t^{\alpha+\gamma}\right) \notag \\
   &=&  \frac{\sqrt{u}+1}{2\sqrt{u}} \left[\sum_{n=0}^\infty \frac{ (\lambda \sqrt{u}t^{\alpha + \gamma })^{2n}}{(2n)!} \operatorname{E}_{\alpha,1+\gamma/\alpha,\gamma/\alpha}^{(2n)}\left(-\lambda t^{\alpha+\gamma}\right) +\sum_{n=0}^\infty \frac{ (\lambda \sqrt{u}t^{\alpha + \gamma })^{2n+1}}{(2n+1)!} \operatorname{E}_{\alpha,1+\gamma/\alpha,\gamma/\alpha}^{(2n+1)}\left(-\lambda t^{\alpha+\gamma}\right)\right] + \notag  \\
&+&\frac{\sqrt{u}-1}{2\sqrt{u}} \left[\sum_{n=0}^\infty \frac{ (-\lambda \sqrt{u}t^{\alpha + \gamma })^{2n}}{(2n)!} \operatorname{E}_{\alpha,1+\gamma/\alpha,\gamma/\alpha}^{(2n)}\left(-\lambda t^{\alpha+\gamma}\right) +\sum_{n=0}^\infty \frac{ (-\lambda \sqrt{u}t^{\alpha + \gamma })^{2n+1}}{(2n+1)!} \operatorname{E}_{\alpha,1+\gamma/\alpha,\gamma/\alpha}^{(2n+1)}\left(-\lambda t^{\alpha+\gamma}\right)\right] \notag \\
&=& \sum_{n=0}^\infty u^n \left[ \frac{ (\lambda t^{\alpha + \gamma })^{2n}}{(2n)!} \operatorname{E}_{\alpha,1+\gamma/\alpha,\gamma/\alpha}^{(2n)}\left(-\lambda t^{\alpha+\gamma}\right) + \frac{ (\lambda t^{\alpha + \gamma })^{2n+1}}{(2n+1)!} \operatorname{E}_{\alpha,1+\gamma/\alpha,\gamma/\alpha}^{(2n+1)}\left(-\lambda t^{\alpha+\gamma}\right)\right] \notag
 \end{eqnarray}
 and formula \eqref{solodd} follows.

The initial condition $\left.\operatorname{\mathcal{D}}^{(\alpha,\gamma)}_tp(0,t)\right\vert_{t=0}=0$ is satisfied, as can be checked by considering eq.\eqref{incon2}, for $t=0$ and for $u=0$.

  \end{proof}

\begin{remark}
As a consequence of the previous theorem, it is evident from equation \eqref{solodd} that the solution to equation \eqref{fde.4}  can be interpreted as the p.m.f. of a process $\overline{\mathcal{N}}_{\alpha,\gamma}^L:=\left\{\overline{\mathcal{N}}_{\alpha, \gamma}^L(t)\right\}_{t \geq 0}$ that jumps upward at the even-order events of $\mathcal{N}_{\alpha, \gamma}^L$, while the probability of the
successive odd-indexed events is added to that of the previous ones.

On the other hand, by recalling Def. \ref{defren2} and considering equation \eqref{eqg} with $u=0$, we can see that  $\overline{p}^L_{\alpha,\gamma}(0,t)=G(0,t)=\mathbb{P}(\overline{U}^{(\alpha,\gamma)}>t)$.
However, although they share the first interarrival time distribution, it can be proved that $\left\{ \overline{p}^L_{\alpha, \gamma}(n;t)\right\}_{n \in \mathbb{N}_0}$ does not coincide with the p.m.f. of the renewal process $\overline{\mathcal{N}}_{\alpha, \gamma}$ defined in Def. \ref{defren2}, as a consequence of Remark \ref{pro}.

The expected value of the process $\overline{\mathcal{N}}_{\alpha,\gamma}^L$ can be obtained  from the probability generating function  \eqref{soldue}, as follows:
\begin{eqnarray}\label{over}
 \mathbb{E}\overline{\mathcal{N}}^L_{\alpha, \gamma}(t)&=&\frac{d}{du}G(u,t) \vert_{u=1}=\frac{\lambda t^{\alpha+\gamma}}{2}\frac{\Gamma(\gamma+1)}{\Gamma(\gamma+\alpha+1)}+\frac{1}{4}\sum_{n=1}^\infty c_n \left(-2\lambda t ^{\alpha+\gamma}\right)^n ,
\end{eqnarray}
for $c_n$ given in \eqref{cn}.
%\[
%c_n:=\prod_{j=0}^{n-1}\frac{\Gamma(j(\gamma+\alpha)+\gamma+1)}{\Gamma((j+1)(\gamma+\alpha)+1)}.
%\]
In the special case where $\gamma=0$, we obtain from \eqref{over} the result given in \cite{BEG}, eq.(3.25).
\end{remark}

\subsubsection{Case $b \neq (a/2)^2$}
We now consider the following model that can be defined starting from the first case  analyzed in Theorem \ref{thm3}, i.e. for $b\neq (a/2)^2$. In order to have real and positive constants $K_1,K_2$ and a well-defined survival probability for the interarrival times, we assume hereafter that %$a>1$ and
$0<b< a^2/4$.

\begin{definition}\label{deftilde}
    Let $\alpha \in (0,1)$, $\alpha + \gamma> 0$, and let $\Hat{\mathcal{N}}_{\alpha, \gamma}:=\left\{\Hat{\mathcal{N}}_{\alpha,\gamma}(t) \right\}_{t \geq 0}$ be the renewal process whose  interarrival times have survival probability $\mathbb{P}(\Hat{U}^{(\alpha, \gamma)}>t)$ satisfying
    \begin{equation}
\label{Hat1}
\left(\operatorname{\mathcal{D}}^{(\alpha,\gamma)}_t\right)^2 f(t) +a
\operatorname{\mathcal{D}}^{(\alpha,\gamma)}_t f(t) +
b f(t) = 0 , \qquad t > 0,
\end{equation}
for $0<b< a^2/4$, and under the conditions $f(0)=1$ and $\left.\operatorname{\mathcal{D}}^{(\alpha,\gamma)}_tf(t)\right\vert_{t=0}=-\lambda$, with $\lambda \in (\eta_1,\eta_2)$, for $\eta_1$, $\eta_2$ given in \eqref{eta1eta2}.

   \end{definition}
It is a straightforward consequence of Theorem \ref{thm3} that the survival probability of the interarrival times is equal to
\begin{equation}
\mathbb{P}(\Hat{U}^{(\alpha,\gamma)}>t)=K \, \operatorname{E}_{\alpha,1+\gamma/\alpha,\gamma/\alpha}\left(-\eta_1 t^{\alpha+\gamma}\right) +
(1-K) \, \operatorname{E}_{\alpha,1+\gamma/\alpha,\gamma/\alpha}\left( -\eta_2 t^{\alpha+\gamma}\right), \label{Hat4}
\end{equation}
and $K$ belongs to $(0,1)$, as a consequence of the assumptions on $a,b$ and $\lambda$. Indeed, the initial condition $f(0)=1$ implies that $K_1+K_2=1$ in \eqref{sol}, while, by considering \eqref{fde.1}, we get, from the other initial condition, that
\begin{equation}
\left.\operatorname{\mathcal{D}}^{(\alpha,\gamma)}_t\mathbb{P}(\Hat{U}^{(\alpha,\gamma)}>t)\right\vert_{t=0}=-K\eta_1-(1-K)\eta_2=-\lambda<0,
\end{equation}
for $\lambda \in (\eta_1,\eta_2)$.

Therefore, formula \eqref{Hat4} gives a proper survival function since it is a non-negative, decreasing function tending to zero as $t \to \infty$, as a linear combination of functions enjoying these properties.

\paragraph{Stochastic representation}
As a consequence of \eqref{Hat4}, we have that the density function of the interarrival times $\Hat{U}^{(\alpha,\gamma)}$ is given by $f_{\Hat{U}^{(\alpha,\gamma)}}(\cdot)=Kf_{U_{\eta_1}^{(\alpha,\gamma)}}(\cdot)+(1-K) f_{U_{\eta_2}^{(\alpha,\gamma)}})(\cdot)$, where $U_{\eta_i}^{(\alpha,\gamma)}$ is the interarrival time of $\mathcal{N}_{\alpha,\gamma}$, with parameter $\eta_i$, for $i=1,2$.
%instead of being the convolution of $U^{(\alpha,\gamma)}$ as for the other second-order renewal process.

Moreover, the expected value reads $\mathbb{E}\Hat{U}^{\alpha, \gamma}=K\mathbb{E}{U}_{\eta_1}^{\alpha, \gamma}+(1-K)\mathbb{E}{U}_{\eta_2}^{\alpha, \gamma}$ and thus displays the same behavior of all the other cases, being infinite, for $\alpha+\gamma<1$, and finite otherwise.

\paragraph{Alternative (non-renewal) approach}
Analogously to the previous sub-section, we consider the second-order generalization of equation \eqref{fde.3} with generating function of the solution satisfying \eqref{Hat1}, under the assumption that $0<b<a^2/4$. In this case, we introduce also a second-order difference-operator (acting on $n$), i.e. $\Delta^2$, where $\Delta u(n):=u(n)-u(n-1)$.

\begin{theorem}\label{thmHat}
The solution to the following equation
\begin{equation}
\label{Hat2}
(\operatorname{\mathcal{D}}^{(\alpha,\gamma)}_t)^2 p_{\alpha, \gamma}(n;t)+a
\mathcal{D}^{(\alpha,\gamma)}_t \Delta p_{\alpha, \gamma}(n;t) +b \Delta^2 p_{\alpha, \gamma}(n;t) = 0,
\end{equation}
 with initial conditions $p_{\alpha, \gamma}(n;0)=\delta_0(n)$, $n\in \mathbb{N}_0$, $t \geq 0$, and
\begin{equation*}
    \left.\operatorname{\mathcal{D}}^{(\alpha,\gamma)}_tp_{\alpha, \gamma}(0;t)\right\vert_{t=0}=-\lambda,
  \qquad
\left.\operatorname{\mathcal{D}}^{(\alpha,\gamma)}_tp_{\alpha, \gamma}(n;t)\right\vert_{t=0}=0, \qquad n=1,2,...
\end{equation*}
reads
\begin{eqnarray}
\Hat{p}^L_{\alpha, \gamma}(n;t)
=K\frac{\left(\eta_1 t^{\alpha + \gamma}\right)^n}{n!}\operatorname{E}_{\alpha,1+\gamma/\alpha,\gamma/\alpha}^{(n)}\left(-\eta_1 t^{\alpha+\gamma}\right)+(1-K)\frac{\left(\eta_2 t^{\alpha + \gamma}\right)^n}{n!}\operatorname{E}_{\alpha,1+\gamma/\alpha,\gamma/\alpha}^{(n)}\left(-\eta_2t^{\alpha+\gamma}\right) \notag
\\
%K p'_{\alpha,\gamma}(n;t)+(1-K) p''_{\alpha,\gamma}(n;t),
\label{Hat3}
\end{eqnarray}
for $n \in \mathbb{N}_0$, $t \geq 0$.
%, where $p'_{\alpha,\gamma}(n;t)$ and $p''_{\alpha,\gamma}(n;t)$ represent the p.m.f. of the first-renewal process $\mathcal{N}_{\alpha,\gamma}$, with parameter $\eta_1$ and $\eta_2$, respectively.
\end{theorem}
\begin{proof}
We multiply eq.\eqref{Hat2} by $u^n$ and add over $n=0,1,...$, so that we get the following equation, for the probability generating function:
\[
(\operatorname{\mathcal{D}}^{(\alpha,\gamma)}_t)^2 G(u,t)+a(1-u)
\mathcal{D}^{(\alpha,\gamma)}_t  G(u,t) +b (1-u)^2G(u,t) = 0,
\]
with $G(u,0)=1$ and  $\left.\operatorname{\mathcal{D}}^{(\alpha,\gamma)}_t G(u,t)\right\vert_{t=0}=-\lambda$, which coincides with \eqref{Hat1}, for $u=0$, and is satisfied by
\begin{equation}\notag
G(u,t)=K\, \operatorname{E}_{\alpha,1+\gamma/\alpha,\gamma/\alpha}\left(-\eta_1 (1-u) t^{\alpha+\gamma}\right) +
(1-K) \, \operatorname{E}_{\alpha,1+\gamma/\alpha,\gamma/\alpha}\left( -\eta_2 (1-u)t^{\alpha+\gamma}\right).
\end{equation}

\end{proof}

\begin{remark}
It is evident by \eqref{Hat3} that the following relationship holds

\begin{equation}\Hat{p}^L_{\alpha,\gamma}(n;t)=
K p^{L'}_{\alpha,\gamma}(n;t)+(1-K) p^{L''}_{\alpha,\gamma}(n;t), \qquad
n \in \mathbb{N}_0, t \geq 0,
\end{equation}
where $p^{L'}_{\alpha,\gamma}(n;t)$ and $p^{L''}_{\alpha,\gamma}(n;t)$ represent the p.m.f. of $\mathcal{N}^L_{\alpha,\gamma}$, with parameter $\eta_1$ and $\eta_2$, respectively.

\end{remark}

\section*{Acknowledgments}
Luisa Beghin acknowledges financial support under NRRP, Mission 4, Component 2, Investment 1.1, Call for tender No. 104 published on 2.2.2022 by the Italian MUR, funded by the European Union – NextGenerationEU– Project Title “Non–Markovian Dynamics and Non-local Equations” – 202277N5H9 - CUP: D53D23005670006.

Nikolai Leonenko (NL) would like to thank for support and hospitality during the program “Fractional Differential Equations” and the programs “Uncertainly Quantification and Modeling of Materials” and “Stochastic systems for anomalous diffusion" in Isaac Newton Institute for Mathematical Sciences, Cambridge. The last program  was organized with the support of the Clay Mathematics Institute, of EPSRC (via grants EP/W006227/1 and EP/W00657X/1), of UCL (via the MAPS Visiting Fellowship scheme) and of the Heilbronn Institute for Mathematical Research (for the Sci-Art Contest). Also NL was partially supported under the ARC Discovery Grant DP220101680 (Australia), Croatian Scientific Foundation (HRZZ) grant “Scaling in Stochastic Models” (IP-2022-10-8081), grant FAPESP 22/09201-8 (Brazil) and the Taith Research Mobility grant (Wales, Cardiff University). Also, NL would like to thank University of Rome “La Sapienza” for hospitality  as Visiting  Professor (June 2024) where the paper was initiated.

Jayme Vaz would like to thank the support of FAPESP (process 24/17510-6), the Taith Research Mobility grant (Wales, Cardiff University), Cardiff University and Sapienza University of Roma for the hospitality during the completion of this paper.

\section*{Competing interests}The authors have no conflicts of interest to declare that are relevant to the content of this article.

\appendix

\section{Analysis of the solution of eq.\eqref{fde.1}}
\label{appendix.A}

In order to prove the convergence and uniqueness of the
solution \eqref{sol.fde1} of eq.\eqref{fde.1} we will use
the following

\begin{lem} \label{lemma_extra} There exists $N \in \mathbb{N}$ such that for all $n > N$ the coefficients in the series \eqref{sol.fde1} satisfy
\begin{equation}
\label{lemma.n}
\frac{1}{[\mathsf{n}!]_\alpha^{\alpha+\gamma}} \leq
C_{\alpha,\gamma} \left(\frac{2^{1-\alpha}}{(\alpha+\gamma)^{\alpha}}\right)^n \frac{1}{[(n-1)!]^\alpha}
\end{equation}
where $C_{\alpha,\gamma}$ is a constant depending on $\alpha \in (0,1)$ and $\gamma > -1$.
\end{lem}

\begin{proof}
Let us recall Gautschi's inequality
\begin{equation}
x^{1-\sigma} \leq \frac{\Gamma(x+1)}{\Gamma(x+\sigma)} \leq (1+x)^{1-\sigma} ,
\end{equation}
where $x > 0$ and $\sigma \in (0,1)$. From it we obtain
\begin{equation}
\frac{\Gamma(x)}{\Gamma(x+\sigma)} \leq \frac{(1+1/x)^{1-\sigma}}{x^\sigma} .
\end{equation}
Using this in \eqref{def.n.beta.alpha} and \eqref{def.n.beta.alpha.factorial} we have
\begin{equation}
\label{eq.n.aux}
\frac{1}{[\mathsf{n}!]_\alpha^{\alpha+\gamma}} \leq
\prod_{k=1}^n \frac{\left[1+\frac{1}{1+\gamma+(k-1)(\alpha+\gamma)}\right]^{1-\alpha}}{\left[1+\gamma+(k-1)(\alpha+\gamma)\right]^\alpha} .
\end{equation}
Firstly, let us analyse the numerator on the RHS of \eqref{eq.n.aux}.
Since $\gamma > -1$, we have $1+\gamma +(k-1)(\alpha+\gamma) > 0$ for $k=1,2,\ldots$. However, it is possible that for some small values of $k$ we have $\gamma + (k-1)(\alpha + \gamma) < 0$. Let $N = N(\alpha,\gamma)$ be the greatest integer such that $\gamma + (k-1)(\alpha + \gamma) < 0$.
Then we can write, for $n > N$,
\begin{equation}
\prod_{k=1}^n \left[1+\frac{1}{1+\gamma+(k-1)(\alpha+\gamma)}\right]^{1-\alpha} = C^{(1)}_{\alpha,\gamma} \prod_{k=N+1}^n \left[1+\frac{1}{1+\gamma+(k-1)(\alpha+\gamma)}\right]^{1-\alpha}
\end{equation}
where
\begin{equation}
C^{(1)}_{\alpha,\gamma} = \prod_{k=1}^N \left[1+\frac{1}{1+\gamma + (k-1)(\alpha+\gamma)}\right]^{1-\alpha} .
\end{equation}
For $k > N$ we have $1+\gamma+(k-1)(\alpha+\gamma) \geq 1$ and then
\begin{equation}
\prod_{k=N+1}^n \left[1+\frac{1}{1+\gamma+(k-1)(\alpha+\gamma)}\right]^{1-\alpha} \leq \prod_{k=N+1}^n 2^{1-\alpha} =
2^{(1-\alpha)(n-N)} .
\end{equation}
So we have
\begin{equation}
\label{eq.num.aux}
\prod_{k=1}^n \left[1+\frac{1}{1+\gamma+(k-1)(\alpha+\gamma)}\right]^{1-\alpha} = C^{(2)}_{\alpha,\gamma} \, 2^{n(1-\alpha)}
\end{equation}
where
\begin{equation}
C^{(2)}_{\alpha,\gamma} = 2^{-N(1-\alpha)} C^{(1)}_{\alpha,\gamma} .
\end{equation}
On the other hand, for the denominator of \eqref{eq.n.aux} we have
\begin{equation}
\prod_{k=1}^n \left[ 1+\gamma + (k-1)(\alpha+\gamma)\right]^\alpha =
(\alpha+\gamma)^{n\alpha}\prod_{k=1}^n \left[ (k-1) + \frac{1+\gamma}{\alpha+\gamma}\right]^\alpha .
\end{equation}
Since $\gamma >-1$ and $\alpha + \gamma > 0$, we have
\begin{equation}
\label{eq.den.aux}
\begin{split}
\prod_{k=1}^n \left[ 1+\gamma + (k-1)(\alpha+\gamma)\right]^\alpha & =
(\alpha+\gamma)^{n\alpha} \left(\frac{1+\gamma}{\alpha+\gamma}\right)^\alpha \prod_{k=2}^n \left[ (k-1) + \frac{1+\gamma}{\alpha+\gamma}\right]^\alpha \\[1ex]
& \geq (\alpha+\gamma)^{n \alpha}\left(\frac{1+\gamma}{\alpha+\gamma}\right)^\alpha \left[
\prod_{k=2}^n (k-1)\right]^\alpha \\[1ex]
& \geq (\alpha+\gamma)^{n\alpha}\left(\frac{1+\gamma}{\alpha+\gamma}\right)^\alpha [(n-1)!]^\alpha .
\end{split}
\end{equation}
Finally, using eq.\eqref{eq.num.aux} and eq.\eqref{eq.den.aux} in eq.\eqref{eq.n.aux} we obtain eq.\eqref{lemma.n} with
\begin{equation}
C_{\alpha,\gamma} = \left(\frac{1+\gamma}{\alpha+\gamma}\right)^\alpha C^{(2)}_{\alpha,\gamma} .
\end{equation}
\end{proof}

Let us consider $f(t)$ in eq.\eqref{sol.fde1} and define $a_n(t)$ as
\begin{equation}
a_n(t) = \frac{(-\kappa t^{\alpha + \gamma})^n}{[\mathsf{n}!]_\alpha^{\alpha+\gamma}} .
\end{equation}
Consider an arbitrary $T > 0$. Then for $0< t < T$ and using eq.\eqref{lemma.n} we have
\begin{equation}
|a_n(t)| \leq C_{\alpha,\gamma} \left( \frac{\kappa T^{(\alpha+\gamma)} 2^{1-\alpha}}{(\alpha+\gamma)^\alpha}\right)^n \frac{1}{[(n-1)!]^\alpha}
= M_n .
\end{equation}
But the series $\sum_{n=1}^\infty M_n$ converges since
\begin{equation}
\lim_{n\to \infty}\frac{M_{n+1}}{M_n} = \lim_{n\to\infty}
\frac{\kappa T^{(\alpha+\gamma)} 2^{1-\alpha}}{(\alpha+\gamma)^\alpha}
\frac{1}{n^\alpha} = 0 .
\end{equation}
Then the series in eq.\eqref{sol.fde1} converges
uniformly and absolutely.

In order to prove uniqueness, the usual procedure is to transform the differential equation into an integral equation. For eq.\eqref{fde.1} with initial condition $f(0) = f_0$, the equivalent integral equation is \cite{Diethelm}
\begin{equation}
\label{equiv.int}
f(t) = f_0 - \frac{\kappa}{\Gamma(\alpha)}\int_0^t (t-s)^{\alpha-1} s^\gamma f(s) \, ds .
\end{equation}
Although the uniqueness of solutions is already discussed in \cite{Diethelm}, we will discuss it here using the result of Lemma~\ref{lemma_extra}. Let $u(t) = |f_1(t) - f_2(t)|$, where $f_1(t)$ and $f_2(t)$ are two solutions of eq.\eqref{equiv.int}. Then
\begin{equation}
u(t) \leq \mathcal{K}[u](t) = \frac{\kappa}{\Gamma(\alpha)}
\int_0^t (t-s)^{\alpha-1} s^\gamma u(s)\, ds .
\end{equation}
By iteration we have
\begin{equation}
u \leq \mathcal{K}[u] \leq \mathcal{K}^2[u] \leq \mathcal{K}^3[u] \leq \cdots
\end{equation}
Let us assume the solutions are in the class of absolutely continuous functions and denote
\begin{equation}
M = \underset{0\leq s \leq t}{\operatorname{max}} u(s) .
\end{equation}
Then we have
\begin{equation}
\mathcal{K}[u](t) \leq \frac{\kappa M}{\Gamma(\alpha)} \int_0^t (t-s)^{\alpha-1} s^\gamma\, ds = \kappa M t^{\alpha + \gamma} \frac{\Gamma(\gamma+1)}{\Gamma(\alpha + \gamma+1)} ,
\end{equation}
where we use the definition of the
Beta function to evaluate the integral. It is not difficult to see, using again the Beta function,
that
\begin{equation}
u(t) \leq \mathcal{K}^n[u](t) = \kappa^n M t^{n(\alpha+\gamma)}
\frac{1}{[\mathsf{n}!]_\alpha^{\alpha+\gamma}}, \qquad n = 1,2,\ldots .
\end{equation}
From Lemma~\ref{lemma_extra} and the arbitrariness of $n$, it follows that $u(t) = 0$.

\begin{remark}
The uniqueness of the solution of the second-order equation~\eqref{eq.2} with initial conditions $f(0) = f_0$ and $\mathcal{D}_t^{(\alpha,\gamma)}f(0) = f_0^\prime$ follows from its factorization as a system of two first-order problems of the above kind, that is, $\mathcal{D}_t^{(\alpha,\gamma)} g(t) + \mu_1 g(t) = 0$ with condition $g(0) = f_0^\prime + \mu_2 f_0$ and $\mathcal{D}_t^{(\alpha,\gamma)} f(t) + \mu_2 f(t) = g(t)$ with condition $f(0) = f_0$, where $\mu_1 + \mu_2 = a$ and $\mu_1 \mu_2 = b$ in eq.\eqref{eq.2}.
\end{remark}

\section{Convergence of the KS' Laplace transform}

In order to analyse the convergence of the integral along $\mathcal{C}$, we will consider
the limit $R \to \infty$ for the integral along $(c-iR,c+iR)$.
However, since the integrand is a holomorphic function in the
entire plane except at the poles described above,  we deform, as usual, the line segment
$(c-iR, c+iR))$  into the contour $\mathcal{C}_R^- \, \cup \,  \mathcal{C}_c  \, \cup \,  \mathcal{C}_R^+ $, where
$\mathcal{C}_R^- = (-iR,-\epsilon)$, $\mathcal{C}_R^+ = (\epsilon,iR)$, and
$\mathcal{C}_c$ is the arc $\{s \in \mathbb{C}\,|\,
|s| = c, -\pi/2 \leq \operatorname{arg}(s) \leq \pi/2\}$
encircling the pole $s=c_0$ from the right. Since we are only
interested in the issue of convergence, we only need to analyse
the integrals along $\mathcal{C}_R^+$ and $\mathcal{C}_R^-$.

Writing $s = R\operatorname{e}^{i\theta}$ and using the Stirling formula for $\Gamma(z)$ and  eq.\eqref{ap.B.Stirling.G.tau} for $G(z;\tau)$, we have
\begin{equation}
\begin{aligned}
& \log|\Gamma(\sigma+s/\nu)| = \nu^{-1}\cos\theta \, R \ln{R} - \nu^{-1}[\theta\sin\theta + \cos\theta(1+\ln\nu)]R
+ (\sigma-1/2)\ln{R} + \mathcal{O}(1) , \\[1ex]
& \log|\Gamma(\sigma-s/\nu)| = -\nu^{-1}\cos\theta \,R \ln{R} + \nu^{-1}[\theta\sin\theta + \cos\theta(1+\ln\nu)-\pi\sin|\theta|]R \\[1ex]
& \phantom{\log|\Gamma(\sigma-s/\nu)| =} + (\sigma-1/2)\ln{R} + \mathcal{O}(1), \\[1ex]
& \log|(\lambda^{-1/\nu} z)^{-s}| = \left(-\cos\theta\ln|\lambda^{-1\nu}z| + \sin\theta \operatorname{arg}(z) \right)\, R , \\[1ex]
& \log|G(\sigma+s/\nu;\tau)| =   \left(\nu^{-2} A_2(\tau,\sigma)\cos{2\theta}\right) R^2\log{R} + \left(\nu^{-1} A_1(\tau,\sigma)\cos\theta\right) R\log{R} \\[1ex]
& \phantom{\log|G(\sigma+s/\nu;\tau)| = } +A_0(\tau,\sigma) \log{R} + \nu^{-2} \left[B_2(\tau,\sigma)\cos{2\theta} -
A_2(\tau,\sigma) (\theta\sin{2\theta}+\ln\nu)\right] R^2 \\[1ex]
& \phantom{\log|G(\sigma+s/\nu;\tau)| = }  +\nu^{-1}\left[B_1(\tau,\sigma)\cos\theta-A_1(\tau,\sigma)(\theta\sin\theta +\cos\theta \ln\nu)\right] R
 + \mathcal{O}(1) ,
\end{aligned}
\end{equation}
where
\begin{align*}
& A_2(\tau,\sigma) = \frac{1}{2\tau} , \\
& A_1(\tau,\sigma) = \frac{\sigma}{\tau} -\frac{1}{2}\left(1+\frac{1}{\tau}\right), \\
& A_0(\tau,\sigma) =  \frac{\sigma^2}{2\tau} - \frac{\sigma }{2}\left(1+\frac{1}{\tau}\right) + \frac{\tau}{12} + \frac{1}{4} + \frac{1}{12\tau} , \\
& B_2(\tau,\sigma) = -\frac{1}{2\tau}\left(\frac{3}{2} + \log\tau\right) , \\
& B_1(\tau,\sigma) =  \frac{\sigma}{2\tau} -  \frac{\sigma}{\tau}\left(\frac{3}{2} + \log\tau\right)  + \frac{1}{2}\left(\left(1+\frac{1}{\tau}\right)(1+\log\tau) + \log{2\pi}\right) .
\end{align*}
Thus, we can write, for the integrand in eq.\eqref{lem2}, that
\begin{equation}
\label{converg.aux}
\ln\left|\left(\lambda^{-1/\nu}z\right)^{-s} \Theta(s)\right| =
C_0 \, R\ln{R} + C_1 \, R + \mathcal{O}(\ln{R}),
\end{equation}
with
\begin{equation}
C_0 = \left(1- \nu^{-1} a\right)\cos\theta
\end{equation}
and
\begin{equation}
\begin{aligned}
C_1 = & \cos\theta\left[-\ln|\lambda^{-1/\nu}z| + \nu^{-1}(a-\nu+1+ \ln{(\nu\tau)}\right] \\[1ex]
& + \sin\theta\operatorname{arg}(z) + (\nu^{-1}a-1)\theta \sin\theta-\nu^{-1}\pi \sin|\theta| .
\end{aligned}
\end{equation}

Let us see what happens with the integrals along $\mathcal{C}_R^+$ and $\mathcal{C}_R^-$.
Along $\mathcal{C}_R^+$ we have $\theta = \pi/2$, and, from eq.\eqref{converg.aux}, we have
\begin{equation}
\ln\left|\left(\lambda^{-1/\nu}z\right)^{-s} \Theta(s)\right| \bigg|_{\mathcal{C}_R^+} = \left[\arg{(z)} +
\left(\frac{a-2}{\nu}-1\right)\frac{\pi}{2}  \right] R + \mathcal{O}(\log{R}) .
\end{equation}
Thus, if
\begin{equation}
\operatorname{arg}z < \left[1 + \frac{(2-a)}{\nu}\right]\frac{\pi}{2},
\end{equation}
the integral along $\mathcal{C}_R^+$ decays exponentially, as $R \to \infty$.
On the other hand, along $\mathcal{C}_R^-$ we have $\theta = -\pi/2$, and so
\begin{equation}
\ln\left|\left(\lambda^{-1/\nu}z\right)^{-s} \Theta(s)\right| \bigg|_{\mathcal{C}_R^+} = \left[-\arg{(z)} +
\left(\frac{a-2}{\nu}-1\right)\frac{\pi}{2}  \right] R + \mathcal{O}(\log{R}) .
\end{equation}
Therefore, if
\begin{equation}
\operatorname{arg}z > -\left[1 + \frac{(2-a)}{\nu}\right]\frac{\pi}{2},
\end{equation}
the integral along $\mathcal{C}_R^-$ decays exponentially, as $R \to \infty$.
Consequently, the integral along $\mathcal{C}$ converges in the sector
\begin{equation}
|\operatorname{arg}z| <\left[1 + \frac{(2-a)}{\nu}\right]\frac{\pi}{2}.
\end{equation}

Let us suppose $0 < a < 2$; then $(2-a)/\nu > 0$ and
we conclude that eq.\eqref{lem2} converges for $|\operatorname{arg}z| < \pi/2$, that is, for $\operatorname{Re} z > 0$.


\begin{thebibliography}{99}

\bibitem{ALE} G. Aletti, N. Leonenko, E. Merzbach, Fractional Poisson fields and martingales, \textit{J. Stat. Phys.},  \textbf{170}, 700–730 (2018).


\bibitem{Alexanian}  S. Alexanian and A. Kuznetsov, On the Barnes double gamma function, \textit{Integral Transforms Spec. Funct.}, \textbf{34}, 891-914 (2023).

\bibitem{Genesis} E.W. Barnes, The genesis of the double gamma functions,
 \textit{Proc. Lond. Math. Soc.}, \textbf{1}(1),
358-381 (1899).


\bibitem{GAR} L. Beghin, L. Cristofaro, R. Garrappa, Renewal processes linked to fractional relaxation equations with variable order, \textit{J. Math. Anal. Appl.}, \textbf{531}, 127795 (2024).

\bibitem{BEG2009} L. Beghin, E. Orsingher, Fractional Poisson processes and related planar random
motions, \textit{Electronic J. Probab.}, 14 (61), 1790-1826 (2009).

\bibitem{BLPV}  L. Beghin, N. Leonenko, I. Papi\'{c}, J. Vaz,
Stretched non-local Pearson diffusions, \texttt{arxiv:}2505.07024, submitted, 1-34 (2025).

\bibitem{BEG} L. Beghin, E. Orsingher, Poisson-type processes governed by fractional and higher-order recursive differential equations, \textit{Electronic J. Probab.}, \textbf{15} (22), 684-709 (2010).

\bibitem{Fibonacci} A.T. Benjamin, J.J. Quinn, \textit{Proofs that Really Count: The Art of Combinatorial Proof}, Mathematical Association of America (2003).


\bibitem{Simon} L. Boudabsa, T. Simon, Some properties of the
Kilbas-Saigo function,
\textit{Mathematics}, \textbf{9} (3), 217 (2021).

\bibitem{BOU} L. Boudabsa, T. Simon, P. Vallois, Fractional extreme distributions,
\textit{Electron. J. Probab.}, \textbf{25}, no. 115, 1–20 (2020).


\bibitem{Capelas1} E. Capelas de Oliveira, F. Mainardi, J. Vaz, Fractional
models of anomalous relaxation based on the Kilbas and Saigo function, \textit{Meccanica},
\textbf{49},~2049-2060~(2014).

\bibitem{Capelas2} E. Capelas de Oliveira, S. Jarosz, J. Vaz, Fractional calculus via Laplace transform and its application in relaxation processes, \textit{Commun. Nonlinear Sci. Numer. Simulat.}, \textbf{69}, 58-72 (2019).

\bibitem{SEFC2nd} E. Capelas de Oliveira, J. Vaz, \textit{Solved Exercises in Fractional
Calculus,} 2nd. ed., Springer Nature Switzerland AG (2025).



\bibitem{CHA} M. Chamberland, A. Straub, On gamma quotients and infinite products, \textit{Advances in Applied Mathematics}, \textbf{51} (5), 546-562 (2013).


\bibitem{Diethelm} K. Diethelm, \textit{The Analysis of Fractional Differential Equations}, Springer-Verlag Berlin Heidelberg (2010).

\bibitem{Gautschi} W. Gautschi, Some elementary inequalities relating to the gamma and incomplete gamma function,\textit{ Journal of Mathematics and Physics}, \textbf{3}, 77-81 (1959).

\bibitem{GER} T. Gergely, I.I. Yezhow, On a construction of ordinary Poisson processes
and their modelling, \textit{Z. Wahrscheinlichkeitstheorie verw. Geb.,} \textbf{27}, 215-232 (1973).

%\bibitem{HEY} C.C. Heyde, Y. Yang, On defining long-range dependence, \emph{J. Appl. Prob.}, \textbf{34}, 939-944 (1997).


\bibitem{Hoggatt} V.E. Hoggatt Jr.,  C.T. Long, Divisibility properties of generalized Fibonacci polynomials, \textit{Fibonacci Quart.} \textbf{12} 113-120, (1974).

%\bibitem{Karlin} S. Karlin, H. M. Taylor, \textit{A Second Course in Stochastic Processes}, Academic Press (1981).

\bibitem{KS.book}A.A. Kilbas, M. Saigo, \textit{$H$-Transforms. Theory and Applications}, in
  Analytical Methods and Special Functions,
  vol. 9, Chapman \& Hall/CRC, Boca Raton (2004).

  \bibitem{KIL}A.A. Kilbas, H.M. Srivastava, J.J. Trujillo, \emph{Theory and
Applications of Fractional Differential Equations}, vol. 204 of
North-Holland Mathematics Studies, Elsevier Science B.V., Amsterdam (2006).

\bibitem{KOC} A.N. Kochubei, General fractional calculus, evolution equations and renewal processes, \textit{Integral Equ. Oper. Theory,} \textbf{71} (4), 583–600 (2011).

\bibitem{LAS1} N. Laskin, Fractional Poisson process, \textit{Commun. Nonlinear Sci. Numer. Simulat.}, \textbf{8}, 201–213 (2003).

\bibitem{LAS} N.Laskin, Fractional non-homogeneous counting process, preprint, arXiv:2312.17389v1 (2023).

\bibitem{LAS2} N.Laskin, A new approach to constructing probability distributions of fractional counting process, \textit{Chaos Solitons Fract.}, 186 (2024), 115268.



\bibitem{lawrie} J. B. Lawrie and A. C. King, ``Exact solution to a class of
functional difference equations with application to a moving contact line flow'',
European Journal of Applied Mathematics \textbf{5}, 141-157 (1994).

\bibitem{LEO} N. Leonenko, M. M. Meerschaert, R.L. Schilling, A. Sikorskii,
Correlation structure of time-changed L\'{e}vy
processes,
\textit{Commun. Appl. Ind. Math.} 1-27 (2014).

%\bibitem{Pearson} N. Leonenko, M. M. Meerschaert, A. Sikorskii,
%Fractional Pearson diffusions,
%\textit{Journ. Mathem. Anal. Appl.,} \textbf{403}, 532–546 (2013).


\bibitem{SCA} N. Leonenko, E. Scalas, M. Trinh, Limit theorems for the fractional non-homogeneous Poisson process, \textit{J. Appl. Probab.}, \textbf{56}, (1), 246–264 (2019).

\bibitem{LET} J. Letemplier, T. Simon, On the law of homogeneous stable functionals, \textit{ESAIM
Probab. Stat.}, 23, 82–111 (2019).

\bibitem{MAI} F. Mainardi, R. Gorenflo, E.Scalas, A fractional generalization
of the Poisson processes, \textit{Vietnam J. Math.}, \textbf{32}, 53-64 (2004).

\bibitem{MEE}  M. Meerschaert, E. Nane, P. Vellaisamy, The fractional Poisson process and
the inverse stable subordinator. \textit{Electron. J. Probab.}, no. 16, 1600–1620 (2011).

\bibitem{MEE2} M.M. Meerschaert, H.P. Scheffler, Limit theorems for continuous-time random walks with infinite mean waiting times, \textit{J. Appl. Probab.}, \textbf{41} (3), 623-638 (2004).

%\bibitem{Meerschaert}  M. Meerschaert, A. Sikorskii, \textit{Stochastic Models for Fractional Calculus}, Walter de Gruyter GmbH, Berlin (2019).

%\bibitem{Nikiforov} A. F. Nikiforov, V. B. Uvarov, \textit{Special Functions of Mathematical Physics: A Unified Introduction with Applications}, Springer Basel (1988).

\bibitem{ORS} E. Orsingher, L. Beghin, Time-fractional telegraph equations and telegraph processes with Brownian time, \textit{Probab. Theory Relat. Fields}, \textbf{128}, 141–160 (2004).

\bibitem{POD} I. Podlubny, \textit{Fractional Differential Equations}, Academic Press, London (1999).

\bibitem{POV} Y. Povstenko, Ostoja-Starzewski, M.: Fractional telegraph equation under moving time-harmonic impact, \textit{Int. J. Heat Mass Transf.}, 182, 121958 (2022).

\bibitem{SCH} R.L. Schilling, R. Song, Z. Vondracek, \emph{Bernstein
Functions: Theory and Applications}, 37, De Gruyter Studies in Mathematics
Series, Berlin (2010).

\bibitem{SUY} J.H. Suyono, A method for computing the autocovariance of renewal
 processes, \textit{J. Korean Stat. Soc.} \textbf{47}, 491–508 (2018).

 \bibitem{VAZ} J. Vaz, E. Capelas de Oliveira, On fractional differential equations, dimensional analysis, and the double gamma function, Nonlinear Dyn. \textbf{113},
 34305-34320 (2025).

\bibitem{Vieira}  N. Vieira, M. Ferreira and M.M. Rodrigues, Time-fractional
telegraph equation with $\psi$-Hilfer derivatives, Chaos, Solitons and Fractals
\textbf{162}, 112276 (2022).

\end{thebibliography}
\end{document}